\documentclass[smallextended]{svjour3}

\usepackage{amsmath,amsfonts,amssymb}
\usepackage{xcolor,bm,url}
\usepackage{booktabs}
\usepackage{array}
\usepackage{tikz}
\usepackage{xspace}
\usepackage[hidelinks]{hyperref}
\hypersetup{
    colorlinks=false,
    urlbordercolor=white,
    breaklinks=true
}

\newtheorem{assumption}[theorem]{Assumption}

\DeclareMathOperator{\diag}{diag}

\DeclareMathOperator{\nullspace}{Ker}
\DeclareMathOperator{\rangespace}{range}
\DeclareMathOperator{\inertia}{inertia}
\newcommand{\ldlt}{$\mathrm{LDL^T}$\xspace}
\newcommand{\lblt}{$\mathrm{LBL^T}$\xspace}
\newcommand{\llt}{\text{Cholesky}\xspace}
\newcommand{\lu}{$\mathrm{LU}$\xspace}
\newcommand{\CG}{\textsc{Cg}\xspace}

\newcommand{\cond}{\kappa_2}
\newcommand{\epstol}{\mathbf{u}}
\newcommand{\cactive}{\mathcal{B}}
\newcommand{\cinactive}{\mathcal{N}}

\newcommand{\add}[1]{#1}

\title{
 Condensed Interior-Point Methods for Scalable Nonlinear Programming on GPUs
}
\author{François Pacaud \and
Sungho Shin \and
Alexis Montoison \and
Michel Schanen \and
Mihai Anitescu
}
\date{\today}

\begin{document}
\maketitle

\begin{abstract}
This paper explores two variants of condensed-space interior-point methods designed for GPUs—HyKKT and LiftedKKT—by analyzing their numerical properties through error analysis and assessing their real-world performance via extensive numerical experiments with a fully GPU-resident software implementation. Traditional implementations of interior-point methods (IPMs) involve solving indefinite augmented KKT systems repeatedly by utilizing direct sparse solvers based on the LBL$^\top$ factorization with sophisticated numerical pivoting strategies. While this method achieves high performance and robustness on CPUs, the serial nature of numerical pivoting presents challenges for effective implementation on GPUs. Recently, multiple condensed-space IPM strategies have emerged to address this issue by transforming the KKT system into a symmetric positive-definite matrix, which is more suitable for factorization on GPUs. However, these methods have been demonstrated only on optimal power flow instances, and their numerical properties are currently inadequately understood. In this study, we demonstrate that although the condensed systems show increased ill-conditioning, the inherent structures of the condensed KKT system effectively counterbalance potential accuracy loss in the IPM. Furthermore, we provide numerical results that thoroughly assess the capabilities of a fully GPU-resident nonlinear programming software stack, comprising MadNLP (a filter line-search IPM solver), cuDSS (a direct sparse solver leveraging Cholesky factorization), and ExaModels (a modeling framework), by benchmarking their performance against the pglib-opf and CUTEst libraries. Our findings suggest that the GPU framework holds promise for solving highly sparse large-scale nonlinear programs, such as optimal power flow instances, except for diminished robustness and limited speedups for edge cases observed within CUTEst instances.
\end{abstract}


\section{Introduction}
General-purpose GPU computing has emerged as a fundamental component of scientific computing and machine learning, particularly in its ability to train and deploy large-scale artificial intelligence (AI) models. As a computing platform, modern GPU hardware offers two major advantages: (1) substantial parallel computing power and (2) improved power efficiency.

Despite these achievements in scientific computing and machine learning, GPUs have been underutilized in classical mathematical programming, including linear programming (LP) and nonlinear programming (NLP). In machine learning applications, parallelization is mostly utilized at the level of low-level matrix operations via libraries such as NVIDIA's {\tt cuBLAS} or AMD's {\tt rocBLAS}. In contrast, most classical mathematical programming applications exhibit limited parallelism at the BLAS operation level due to the inherent sparsity of the underlying problem instances, which prevents straightforward utilization of parallel resources. Although certain categories of sparse matrix operations can be parallelized, the development of efficient sparse matrix kernels on GPUs poses significant challenges, particularly for sparse matrix factorizations. Classical solution methodologies typically depend on direct linear solvers that employ sophisticated numerical pivoting strategies to maintain numerical stability. However, these techniques are inherently sequential, creating difficulties for effective parallelization on GPUs. As a result, until recently, the full potential of GPU computing remained unrealized within the mathematical programming domain.

Recent advancements in GPU computing techniques are beginning to change this landscape. Several significant developments are:

\begin{enumerate}
\item \textbf{Development of mature direct sparse linear solvers}: Recent developments in GPU-based sparse direct solvers, such as NVIDIA's {\tt cuDSS}~\cite{nvidiaNVIDIACuDSSPreview}, have significantly improved the performance of the direct, repetitive solution of sparse linear systems. These solvers now offer performance that rivals or exceeds that of traditional CPU-based solvers, although the types of factorizations currently supported are limited to Cholesky, \ldlt, and LU~\cite{nvidiaNVIDIACuDSSPreview,shin2023accelerating}.
\item \textbf{Progress in automatic differentiation}: GPUs are well-suited for automatic differentiation, a key component of both scientific simulations and machine learning~\cite{jax2018github,enzyme2021}. Efficient implementations for algebraic modeling and derivative evaluation on GPUs have now become available. For large-scale problem instances with highly repetitive patterns, derivative evaluation can be performed at speeds that are orders of magnitude faster than their CPU counterparts~\cite{shin2023accelerating}. Notably, derivative evaluation kernels can be created in a fully automated fashion. This removes the barrier of having to write custom derivative evaluation kernels for each problem instance, which was previously a major bottleneck.
\item \textbf{Rise of GPU-dominated HPC}: With the emergence of exascale supercomputers (e.g., Frontier, Aurora), GPU support has become a core requirement for utilizing these HPC resources in mathematical optimization applications.
\item \textbf{Growing interest in batched optimization}:
  In machine learning, embedding optimization problems into neural networks requires solving many small problems in parallel. This has led to the development of GPU-native batched solvers capable of handling thousands of small problems simultaneously~\cite{amos2017optnet,pineda2022theseus}. Similar ideas have been applied to decomposition methods in large-scale optimization~\cite{kimLeveragingGPUBatching2021}, but extending these methods to large-scale problems remains difficult, as current implementations often rely on dense linear algebra.
\end{enumerate}

Due to this renewed interest in GPU-based optimization, there has been a surge of research in this area, particularly in developing scalable algorithms and software implementations that can harness the power of GPUs. Here, we provide a general overview of the state of the art in GPU-based optimization, highlighting the challenges and opportunities. Recent developments in GPU-based optimization can be broadly categorized into three main approaches: (i) first-order methods, (ii) second-order methods with iterative linear algebra, and (iii) second-order methods with sparse direct solvers. In terms of numerical implementation, the first two approaches rely mostly on repetitive sparse matrix-vector multiplications, while the third approach requires sparse direct solvers.

\paragraph{First-order methods for large-scale optimization.}
Although first-order methods have been extensively used for large-scale optimization, especially in machine learning applications, their application to classical mathematical programming remains limited due to their slow linear convergence rate.
However, for extremely large-scale instances of linear programs, where the numerical factorization of the KKT systems or normal equations is infeasible, first-order methods can be a promising alternative.
Recently, a special type of primal-dual first-order method, known as the primal-dual hybrid gradient (PDHG) method, has been developed to target such extremely large problems.
This method is particularly effective for large-scale linear programs, as it can handle problems with billions of variables and constraints.
This algorithm is implemented within a solver called PDLP, which is adapted to GPU architectures~\cite{lu2023cupdlp,lu2023cupdlp2}.
Recently, this algorithm has been enhanced with the restarted Halpern PDHG method with reflection and a PID-controlled primal weight update strategy, showing improved speed~\cite{luCuPDLPFurtherEnhanced2025}.
Although these methods have faced challenges in obtaining high-precision solutions, they have been shown to outperform even commercial solvers such as Gurobi on extremely large-scale instances of linear programs when fully GPU-accelerated~\cite{lu2023cupdlp,lu2023cupdlp2}.

\paragraph{Second-order methods with iterative linear algebra.}
An alternative to first-order methods is to use second-order methods with iterative linear algebra, such as Krylov subspace methods.
Krylov subspace methods require only matrix-vector products and are inherently parallelizable.
These methods are well-suited for large-scale optimization problems, as they can efficiently solve the KKT systems arising in interior-point methods without requiring full matrix factorizations.
However, most second-order algorithms, such as the primal-dual interior-point method, rely on barrier functions to handle inequality constraints, which can lead to ill-conditioned KKT systems; thus, the iterative linear solvers need to be employed with careful KKT system reformulation and/or preconditioning strategies \cite{curtisNoteImplementationInteriorpoint2012,rodriguezScalablePreconditioningBlockstructured2020}.
Promising results have been reported in convex settings~\cite{schubigerGPUAccelerationADMM2020} by combining operator splitting methods and the preconditioned conjugate gradient.
For nonlinear programming, a strategy based on the Augmented Lagrangian method, interior-point method, and preconditioned conjugate gradient method has been proposed~\cite{cao2016augmented}.

\paragraph{Second-order methods for large-scale optimization.}
The previous two strategies are well suited to large-scale convex optimization, but extending their success to general nonlinear programming remains challenging.
In nonconvex settings, first-order or iterative linear algebra-based implementations often struggle to achieve high accuracy.
For instance, it has been previously reported that ADMM fails to converge reliably below a tolerance of $10^{-3}$ on the AC optimal power flow (OPF) problem~\cite{kimLeveragingGPUBatching2021}.
However, second-order methods with direct linear algebra, when implemented correctly, can offer significantly faster and more robust convergence behavior.
By solving a Newton system at each iteration, they can provide fast local convergence and high robustness, even in the presence of nonconvexities.

Second-order methods with direct linear algebra are substantially more challenging to implement on GPUs.
This is because direct solution methods require the efficient factorization of large, sparse, indefinite Karush–Kuhn–Tucker (KKT) systems.
To the best of our knowledge, no GPU-based direct solver can directly handle the sparse indefinite KKT systems encountered in large-scale nonlinear optimization applications.
Historically, GPU-based sparse solvers have lagged behind their CPU counterparts in both performance and robustness.
Notably, it has been reported that GPU solvers do not offer significant acceleration for power system optimization problems~\cite{swirydowicz2021linear,tasseff2019exploring}.
Although several direct sparse solvers, such as {\tt cuSOLVERRF}, were available on GPUs, they were not widely adopted in practice for solving large-scale nonlinear optimization problems due to their limited performance and robustness.

This situation has changed in recent years with two advances: the development of pivoting-free interior-point methods and the emergence of new GPU-native sparse direct solvers.
First, a KKT system reformulation strategy, known as condensed-space methods, has been proposed to address the challenges of solving indefinite KKT systems within interior-point methods.
There are various forms of such reformulations, including the equality relaxation method (also called LiftedKKT) and the hybrid Golub-Greif method (also called HyKKT; adapted in~\cite{regev2023hykkt}).
Alternatively, null-space methods (also known as reduced-Hessian methods) eliminate equality constraints to reduce the KKT system to a dense problem, which maps well to GPU architectures \cite{pacaud2022condensed}.
These strategies transform the indefinite KKT system into a symmetric positive-definite (SPD) system, which can be solved using Cholesky factorization with static pivoting.

Furthermore, recent advances in GPU-based sparse direct solvers have significantly improved their performance and robustness.
Especially for solving positive-definite or quasi-definite systems, which do not require sophisticated pivoting strategies, GPU solvers now often surpass the performance of CPU-based solvers.
Notably, in November 2023, NVIDIA released {\tt cuDSS}, a new sparse direct solver providing efficient sparse factorizations with significantly improved performance over earlier GPU-based solvers.
These advances, combined with recent developments in automatic differentiation and interior-point methods that avoid numerical pivoting, enable robust and scalable second-order methods to be implemented fully on GPUs.
Coupling {\tt cuDSS} with a GPU-native AD engine allows us to solve large-scale nonlinear problems like OPF up to ten times faster than previous state-of-the-art methods~\cite{shin2023accelerating}.

Nevertheless, condensed-space methods have several known limitations.
Although condensed KKT systems remain sparse in many problem instances, performing the elimination of the inequality constraints often leads to significantly increased nonzero entries in the condensed KKT system
for certain instances.
The worst case is a fully dense constraint, where both the equality relaxation and HyKKT methods mentioned above can result in fully dense condensed KKT systems, rendering the problem prohibitively expensive to solve.
Consequently, condensed-space methods are inherently unsuitable for problems with dense constraints.
Accordingly, when these methods are compared against standard benchmark sets such as CUTEst, they are bound to exhibit varied performance due to the presence of instances with such dense constraints.
Furthermore, due to the nature of condensed KKT systems, the condition number of the condensed KKT system can be significantly larger than that of the original KKT system, which can lead to numerical issues.

\paragraph{A full Software Stack for GPU-Accelerated Nonlinear Programming.}
The state-of-the-art nonlinear optimization software stack consists of three key components:
(i) a high-level modeling interface,
(ii) a nonlinear optimization solver,
(iii) a sparse linear solver.
In the classical CPU-based setting, the most widely used software stack for nonlinear optimization consists of AMPL~\cite{fourer1990ampl} for modeling, Ipopt for the solver, and the HSL library (MA27, MA57, etc.) for sparse factorization.
To fully benefit from the computational power of GPUs, the entire optimization pipeline must be implemented on the GPU.
If any component is not GPU-compatible, it is necessary to introduce costly CPU–GPU communication, similar to double-buffering in MPI-based distributed systems.
This avoids expensive transfers between host (CPU) and device (GPU) memory.
For older generations of the PCIe interface, this can be a major bottleneck.
Thus, to achieve high performance, it is essential that all components of the optimization pipeline are GPU-native and operate directly in GPU memory.
The GPU-native software stack used in this paper is provided by MadNLP, ExaModels, and {\tt cuDSS}.
This paper focuses on the first component, the nonlinear optimization solver, while the other two components are briefly described below.

ExaModels is an algebraic modeling system designed for efficient model function and derivative evaluations on GPUs~\cite{shin2023accelerating}.
ExaModels provides a symbolic modeling front-end where users can define their optimization problems in a structure-revealing fashion, similar to the constraint template implemented in Gravity~\cite{hijazi2018gravity}.
Unlike classical modeling interfaces such as AMPL, JuMP, and CasADi, which provide flexible modeling features, ExaModels follows a different philosophy: it enforces strict compliance with iterator-based modeling syntax, allowing the modeling system to store the model equations in a structured fashion.
This ensures that the model's function and derivative evaluations can be implemented as parallelizable GPU kernels.
For structured model functions, ExaModels applies reverse-mode automatic differentiation to compute the derivatives, which are then executed over GPU-resident data arrays.
To the best of our knowledge, ExaModels is currently the only available GPU-native sparse modeling interface that provides automatic differentiation capabilities for nonlinear programming beyond unconstrained problems in machine learning.

NVIDIA's {\tt cuDSS} library~\cite{nvidiaNVIDIACuDSSPreview} provides a high-performance sparse direct solver.
The library supports Cholesky, LDL$^\top$, and LU factorizations, enabling the efficient solution of symmetric positive-definite and quasi-definite systems.
When the hybrid host/device execution mode is disabled, reordering (a major part of the analysis phase) is executed on the host, while symbolic factorization (another part of the analysis phase), numerical factorization, and solve are executed on the GPU.
Since ordering is inherently sequential, it is precomputed once on the CPU and transferred to the GPU without any impact on performance.
Partial pivoting strategies are available, but the 2x2 pivoting strategy --- commonly employed within sparse indefinite solvers -- is not supported, thus limiting their direct usage within classical augmented system-based interior-point methods.
CUDSS.jl, a community-maintained Julia wrapper for {\tt cuDSS}, provides a user-friendly interface to the library, enabling seamless integration with Julia-based optimization solvers such as MadNLP.

\subsection{Contributions}
In this article, we aim to theoretically and empirically assess the capabilities of GPU-accelerated solution strategies for large-scale sparse nonlinear programs.
Although we utilize the full GPU-resident software stack, we focus on the nonlinear optimization solver component, which is based on a filter line-search interior point method with two variants of condensed-space KKT system reformulation strategies: HyKKT~\cite{regev2023hykkt} and LiftedKKT~\cite{shin2023accelerating}.
Our main contributions are summarized below:
\begin{enumerate}
\item We perform a rigorous error analysis of condensed-space IPM methods for the two methods under study: HyKKT and LiftedKKT. We extend classical results from~\cite{wright1998ill} to show that condensed KKT matrices are structurally ill-conditioned in a controlled manner. This yields accurate error bounds for the Schur complement used in the HyKKT method as well as for the condensed KKT system within LiftedKKT.
\item We provide a fully GPU-resident software stack---consisting of ExaModels (derivative evaluation), MadNLP (interior point solver), and NVIDIA {\tt cuDSS} ---for solving large-scale sparse nonlinear optimization problems, based on two different condensed-space KKT system reformulation strategies: HyKKT and LiftedKKT. Compared to our previous work~\cite{shin2023accelerating}, we (i) have additionally implemented the HyKKT method and (ii) have replaced the linear solver {\tt cusolverRF} used in~\cite{shin2023accelerating} with NVIDIA's new {\tt cuDSS} library. As a result, the interior-point method is executed entirely on the GPU, from model evaluation to KKT system solving.
\item We conduct a comparative study of the two condensed-space methods, analyzing their accuracy and runtime performance. We go beyond traditional OPF test cases by including large-scale problems from the CUTEst benchmark~\cite{gould2015cutest}. Our findings demonstrate that the condensed-space IPM enables a remarkable tenfold acceleration in solving large-scale OPF instances when utilizing the GPU. However, performance outcomes on the CUTEst benchmark exhibit more variability. One of the main reasons for performance variability is that a single dense row can result in a dense condensed KKT system, which can make the factorization of the condensed KKT system significantly more expensive.
\item We demonstrate that condensed-space methods can be beneficial for a wide range of parallel architectures, not limited to GPUs. In fact, the condensed-space strategies presented in this paper are suitable not only for GPUs but also for other parallel architectures, such as multi-core CPUs with shared memory. We provide numerical results of the two condensed-space methods on multi-core CPUs. \add{Our results on the OPF benchmark show that when running on the CPU with Pardiso,
  HyKKT and LiftedKKT achieve a time to solution on par with HSL MA27 on large-scale OPF instances.}
\end{enumerate}

\subsection{Notations}
By default, the norm $\|\cdot\|$ refers to the 2-norm.
We define the conditioning of a matrix $A$ as $\cond(A) = \| A \| \|A^{-1} \|$.
For any real number $x$, we denote by $\widehat{x}$ its floating-point representation.
We denote $\epstol$ as the smallest positive number such that $\widehat{x} \leq (1 + \tau) x$ for $|\tau| < \epstol$.
In double precision, $\epstol = 1.1 \times 10^{-16}$.
We use the following notation to proceed with our error analysis.
For $p \in \mathbb{N}$ and a positive variable $h$:
\begin{itemize}
  \item We write $x = O(h^p)$ if there exists a constant $b > 0$ such that $\| x \| \leq b h^p$;
  \item We write $x = \Omega(h^p)$ if there exists a constant $a > 0$ such that $\| x \| \geq a h^p$;
  \item We write $x = \Theta(h^p)$ if there exist two constants $0 < a < b$ such that $a h^p \leq \| x \| \leq b h^p$.
\end{itemize}

\section{Primal-dual interior-point method}

The interior-point method (IPM) is among the most popular algorithms
for solving nonlinear programs. The algorithm is based on
reformulating the Karush-Kuhn-Tucker (KKT) conditions of the nonlinear program as a smooth
system of nonlinear equations using a homotopy-based method~\cite{nocedal_numerical_2006}.
In a standard implementation, the resulting system is solved iteratively using a Newton method combined with a line-search method for globalization.

This section provides a brief overview of the IPM algorithm and its numerical implementation.
Section \ref{sec:ipm:problem} describes the nonlinear programming problem formulation.
In Section \ref{sec:ipm:kkt}, we detail the Newton step computation within each IPM iteration.

\subsection{Problem formulation and KKT conditions}
\label{sec:ipm:problem}
We are interested in solving the following nonlinear program:
\begin{equation}
  \label{eq:problem}
    \min_{x \in \mathbb{R}^n} \;  f(x)
\quad \text{subject to}\quad
     g(x) = 0 \; , ~ h(x) \leq 0 \; ,
\end{equation}
where $f:\mathbb{R}^n \to \mathbb{R}$ is a real-valued function
encoding the objective, $g: \mathbb{R}^n \to \mathbb{R}^{m_e}$
encoding the equality constraints, and $h: \mathbb{R}^{n} \to
\mathbb{R}^{m_i}$ encoding the inequality constraints.
We assume that the functions $f, g, h$ are smooth
and twice differentiable.

We reformulate \eqref{eq:problem} using non-negative slack variables $s \geq 0$
into the equivalent formulation
\begin{equation}
  \label{eq:slack_problem}
    \min_{x \in \mathbb{R}^n, ~ s \in \mathbb{R}^{m_i}} \;  f(x)
    \quad \text{subject to} \quad
    \left\{
  \begin{aligned}
    & g(x) = 0 \; , ~ h(x) + s = 0 \; , \\
      &  s \geq 0  \; .
  \end{aligned}
  \right.
\end{equation}
In~\eqref{eq:slack_problem}, the inequality constraints
are encoded as variable bounds on the slack variables.
We denote by $y \in \mathbb{R}^{m_e}$ and $z \in \mathbb{R}^{m_i}$ the multipliers associated
with the equality and inequality constraints, respectively.
Similarly, we denote by $v \in \mathbb{R}^{m_i}$ the multipliers associated
with the bounds $s \geq 0$.

Using the dual variables $(y, z, v)$, we define the Lagrangian of \eqref{eq:slack_problem} as
\begin{equation}
  \label{eq:lagrangian}
  L(x, s, y, z, v) = f(x) + y^\top g(x) + z^\top \big(h(x) +s\big)
  - v^\top s \; .
\end{equation}
The KKT conditions of \eqref{eq:slack_problem} are:
\begin{subequations}
  \label{eq:kktconditions}
    \begin{align}
      & \nabla f(x) + \nabla g(x)^\top y + \nabla h(x)^\top z  = 0 \; , \\
      & z - v = 0 \; , \\
      & g(x) = 0 \; , \\
      & h(x) + s = 0 \; , \\
      \label{eq:kktconditions:comps}
      & 0 \leq s \perp v \geq 0 \; .
    \end{align}
\end{subequations}
The notation $s \perp v$ is a shorthand for the complementarity
condition $s_i v_i = 0$ for all $i=1,\cdots, m_i$.
The set of active constraints at a point $x$ is denoted by:
\begin{equation}
  \mathcal{B}(x) := \{ i \in\{ 1, \cdots, m_i\} \; | \; h_i(x) = 0 \} \; .
\end{equation}
The inactive set is the complement $\mathcal{N}(x) := \{1, \cdots, m_i \} \setminus \mathcal{B}(x)$.
We denote by $m_a$ the number of active constraints.
The active Jacobian is defined as $A(x) := \begin{bmatrix} \nabla g(x) \\ \nabla h_{\mathcal{B}}(x) \end{bmatrix} \in \mathbb{R}^{(m_e + m_a) \times n}$.

\subsection{Solving the KKT conditions with the interior-point method}
\label{sec:ipm:kkt}
The interior-point method aims to find a stationary point
satisfying the KKT conditions~\eqref{eq:kktconditions}.
The complementarity constraints \eqref{eq:kktconditions:comps}
render the KKT conditions non-smooth, complicating the solution of
the system~\eqref{eq:kktconditions}.
The IPM uses a homotopy continuation method to solve a simplified
version of \eqref{eq:kktconditions}, parameterized by a barrier
parameter $\mu > 0$~\cite[Chapter 19]{nocedal_numerical_2006}.
For positive $(x, s, v) > 0$, we solve the system
\begin{equation}
  \label{eq:kkt_ipm}
  F_\mu(x, s, y, z, v) =
  \begin{bmatrix}
       \nabla f(x) + \nabla g(x)^\top y + \nabla h(x)^\top z   \\
       z - v  \\
       g(x)  \\
       h(x) + s  \\
       S v - \mu e
  \end{bmatrix} = 0
   \; .
\end{equation}
We introduce in \eqref{eq:kkt_ipm} the diagonal matrix
$S = \diag(s_1, \cdots, s_{m_i})$, along with the vector of ones $e$.
As we drive the barrier parameter $\mu$ to $0$, the solution of the
system $F_\mu(x, s, y, z, v) = 0$ tends toward the solution of the
KKT conditions~\eqref{eq:kktconditions}.

At a fixed parameter $\mu$, the function $F_\mu(\cdot)$ is smooth. Hence, the system \eqref{eq:kkt_ipm} can be solved iteratively
using a sandard Newton method. For a primal-dual iterate
$w_k := (x_k, s_k, y_k, z_k, v_k)$, the next iterate is computed as
$w_{k+1} = w_k + \alpha_k d_k$, where $d_k$ is a descent
direction computed by solving the linear system
\begin{equation}
  \label{eq:newton_step}
  \nabla_w F_{\mu}(w_k) d_k = -F_{\mu}(w_k) \; .
\end{equation}
The step $\alpha_k$ is computed using a line-search algorithm
to ensure that the bounded variables remain positive
at the next primal-dual iterate: $(x_{k+1}, s_{k+1}, v_{k+1}) > 0$.
Once the iterates are sufficiently close to the central path,
the IPM decreases the barrier parameter $\mu$ to obtain a solution closer to
the original KKT conditions~\eqref{eq:kktconditions}.

In the IPM, the bulk of the computational effort lies in the computation of the Newton
step \eqref{eq:newton_step}, which involves assembling the Jacobian
$\nabla_w F_\mu(w_k)$ of the KKT conditions and solving the resulting linear system to compute
the descent direction $d_k$.
By writing out all the blocks, the system in~\eqref{eq:newton_step} expands as the $5 \times 5$
\emph{unreduced KKT system} of the following form:
\begin{equation}
  \label{eq:kkt:unreduced}
  \setlength\arraycolsep{5pt}
  \tag{$K_3$}
  \begin{bmatrix}
    W_k & 0   & G_k^\top         & H_k^\top         & \phantom{-}0 \\
    0   & 0   & 0\phantom{^\top} & I\phantom{^\top} & -I           \\
    G_k & 0   & 0\phantom{^\top} & 0\phantom{^\top} & \phantom{-}0 \\
    H_k & I   & 0\phantom{^\top} & 0\phantom{^\top} & \phantom{-}0 \\
    0   & V_k & 0\phantom{^\top} & 0\phantom{^\top} & \phantom{-}S_k
  \end{bmatrix}
  \begin{bmatrix}
    d_x \\
    d_s \\
    d_y \\
    d_z \\
    d_v
  \end{bmatrix}
  = - \begin{bmatrix}
    \nabla_x L(w_k) \\
       z_k - v_k  \\
       g(x_k)  \\
       h(x_k) + s_k  \\
       S_k v_k - \mu e
  \end{bmatrix} \; ,
\end{equation}
where $W_k = \nabla^2_{x x} L(w_k)$ is the Hessian and
$G_k = \nabla g(x_k)$, $H_k = \nabla h(x_k)$ are the Jacobians.
In addition, we define $S_k$, $U_k$ and $V_k$ as the diagonal matrices built respectively
from the vectors $s_k$, $u_k$ and $v_k$.
Note that~\eqref{eq:kkt:unreduced} can be symmetrized using simple block row and column operations.
Hereafter, we omit the iteration counter $k$ to simplify the notation.

\paragraph{Augmented KKT system.}
It is standard practice to eliminate in \eqref{eq:kkt:unreduced} the blocks associated
with the bound multipliers $v$ and solve the regularized
$4 \times 4$ symmetric system, known as the \emph{augmented KKT system}:
\begin{equation}
  \label{eq:kkt:augmented}
  \tag{$K_2$}
  \setlength\arraycolsep{3pt}
  \begin{bmatrix}
    W + \delta_w I & 0   & \phantom{-}G^\top           & \phantom{-}H^\top           \\
      0       & D_s + \delta_w I  & \phantom{-}0\phantom{^\top} & \phantom{-}I\phantom{^\top} \\
      G       & 0   & -\delta_c I  & \phantom{-}0\phantom{^\top} \\
    H       & I   & \phantom{-}0\phantom{^\top} & -\delta_c I
  \end{bmatrix}
  \begin{bmatrix}
    d_x \\
    d_s \\
    d_y \\
    d_z
  \end{bmatrix}
  = - \begin{bmatrix}
    r_1 \\ r_2 \\ r_3 \\ r_4
  \end{bmatrix} \; ,
\end{equation}
with the diagonal matrix $D_s := S^{-1} V$.
The vectors forming the right-hand sides are
$r_1 := \nabla f(x) + \nabla g(x)^\top y + \nabla h(x)^\top z$,
$r_2 := z - \mu S^{-1} e$,
$r_3 := g(x)$,
$r_4 := h(x) + s$.
Once \eqref{eq:kkt:augmented} is solved, we recover the updates for the bound multipliers via
$d_v = - S^{-1}(V d_s - \mu e) - v$.

Note that we have added regularization terms $\delta_w \geq 0$
and $\delta_c \geq 0$ in \eqref{eq:kkt:augmented} to ensure that the
matrix is invertible.
Without these terms, the augmented KKT system is non-singular
if and only if the Jacobian $J = \begin{bmatrix} G \; &\; 0 \\ H \;&\; I \end{bmatrix}$
has full row rank and the matrix $\begin{bmatrix} W  & 0 \\ 0 & D_s \end{bmatrix}$
projected onto the null-space of the Jacobian $J$ is definite~\cite{benzi2005numerical}.
The condition is satisfied if the inertia (the numbers
of positive, negative and zero eigenvalues) of the matrix~\eqref{eq:kkt:augmented} is $(n + m_i, m_i + m_e, 0)$.
We use the inertia-controlling method introduced in \cite[Section 3.1]{wachter2006implementation}
to regularize the augmented matrix by adding multiple of the identity
on the diagonal of \eqref{eq:kkt:augmented} if the inertia is not equal to $(n+m_i, m_e+m_i, 0)$.

As a result, the system \eqref{eq:kkt:augmented} is usually factorized using
an inertia-revealing \lblt factorization~\cite{duff1983multifrontal}.
Krylov methods are often not competitive for solving~\eqref{eq:kkt:augmented},
because the block diagonal terms $D_s$ become increasingly
ill-conditioned near the solution. Their use in the IPM has been limited to
linear and convex quadratic programming \cite{gondzio-2012} when paired
with a suitable preconditioner. We also refer to \cite{cao2016augmented}
for an efficient implementation of a preconditioned conjugate gradient
on GPU for solving the Newton step in an augmented Lagrangian interior-point
approach.

\paragraph{Condensed KKT system.}
The $4 \times 4$ KKT system \eqref{eq:kkt:augmented} can be further
reduced to a $2 \times 2$ system by eliminating the two blocks
$(d_s, d_z)$ associated with the inequality constraints.
The resulting system is called the \emph{condensed KKT system}:
\begin{equation}
  \label{eq:kkt:condensed}
  \tag{$K_1$}
  \setlength\arraycolsep{3pt}
  \begin{bmatrix}
    K & \phantom{-} G^\top \\
    G & -\delta_c I
  \end{bmatrix}
  \begin{bmatrix}
    d_x \\ d_y
  \end{bmatrix}
  =
  -
  \begin{bmatrix}
    r_1 + H^\top(D_H r_4 - C r_2) \\ r_3
  \end{bmatrix}
  =:
  \begin{bmatrix}
    \bar{r}_1 \\ \bar{r}_2
  \end{bmatrix}
   \; ,
\end{equation}
where $K := W + \delta_w I  + H^\top D_H H$ is the \emph{condensed matrix}
and the two diagonal matrices are:
\begin{equation}
  C := \big(I + \delta_c(D_s + \delta_w I)\big)^{-1} \; , \quad
  D_H := (D_s + \delta_w I) C \; .
\end{equation}
Using the solution of~\eqref{eq:kkt:condensed},
we recover the updates for the slacks and inequality multipliers via
$d_z = -C r_2 + D_H(H d_x + r_4)$ and $d_s = -(D_s + \delta_w I)^{-1}(r_2 + d_z)$.
Using Sylvester's law of inertia, we can prove that
\begin{equation}
  \label{eq:ipm:inertia}
  \inertia(K_2) = (n+m_i, m_e+m_i, 0) \iff
  \inertia(K_1) = (n, m_e, 0) \;.
\end{equation}
This indicates that one can check the inertia of $K_1$ to indirectly check the inertia of $K_2$, which, in turn, can be used to determine the regularization parameters $\delta_w$ and $\delta_c$ needed to ensure that the KKT step $d_k$ is a descent direction.

\paragraph{Iterative refinement.}
Compared to \eqref{eq:kkt:unreduced}, the diagonal matrix $D_s$ introduces additional ill-conditioning in \eqref{eq:kkt:augmented}, which is amplified in the condensed form~\eqref{eq:kkt:condensed}: diagonal elements tend to infinity if a variable approaches its bound and to $0$ if the variable is inactive.
To address the numerical errors arising from such ill-conditioning, most implementations of the IPM employ Richardson iterations on the original system~\eqref{eq:kkt:unreduced} to refine the solution returned by the direct sparse linear solver (see \cite[Section 3.10]{wachter2006implementation}).

\subsection{Discussion}
We have identified three formulations of the KKT systems
appearing at each IPM iteration. The original formulation
\eqref{eq:kkt:unreduced} has better
conditioning than the two alternatives \eqref{eq:kkt:augmented} and
\eqref{eq:kkt:condensed}, but it is much larger.
The second formulation~\eqref{eq:kkt:augmented}, often called the \textit{augmented KKT system}, is
used by default in state-of-the-art nonlinear solvers~\cite{wachter2006implementation,waltz2006interior}.
The system~\eqref{eq:kkt:augmented} is typically factorized using an \lblt factorization: for sparse matrices, the Duff and Reid
multifrontal algorithm~\cite{duff1983multifrontal} is the preferred method (as implemented in the HSL linear solvers MA27 and MA57~\cite{duff2004ma57}).
The condensed KKT system~\eqref{eq:kkt:condensed} is often discarded because
its conditioning is worse
than that of \eqref{eq:kkt:augmented} (implying less accurate solutions).
Additionally, condensation may result in increased fill-in within the condensed system \eqref{eq:kkt:condensed}~\cite[Section 19.3, p.571]{nocedal_numerical_2006}.
In the worst cases, \eqref{eq:kkt:condensed} itself may become fully dense if an inequality row is completely dense.
Consequently, condensation methods are not commonly utilized in practical optimization settings.
To the best of our knowledge, Knitro~\cite{waltz2006interior} is the only solver that supports computing the descent direction using \eqref{eq:kkt:condensed}.

\section{Solving KKT Systems on the GPU}
The GPU has emerged as a prominent computing platform not only for graphics but also for general-purpose computing.
GPUs employ a SIMD architecture that yields excellent throughput for parallelizing small-scale operations.
However, their utility remains limited when algorithms require global communication.
Sparse factorization algorithms, which heavily rely on numerical pivoting, pose significant challenges for implementation on GPUs.
Previous research has demonstrated that GPU-based linear solvers significantly lag behind their CPU counterparts \cite{swirydowicz2021linear,tasseff2019exploring}.
One emerging strategy is to utilize sparse factorization techniques that do not necessitate numerical pivoting \cite{regev2023hykkt,shin2023accelerating} by leveraging the structure of the condensed KKT system \eqref{eq:kkt:condensed}.

We present two alternative methods to solve \eqref{eq:kkt:condensed}.
HyKKT, introduced in \S\ref{sec:kkt:golubgreif}, uses the hybrid strategy of Golub \& Greif~\cite{golub2003solving,regev2023hykkt}.
LiftedKKT~\cite{shin2023accelerating} employs an equality relaxation strategy and is presented in \S\ref{sec:kkt:sckkt}.
Before introducing these methods, we make the following observation, which simplifies the setting for our analysis.
\begin{remark}
  \label{rmk:dualreg}
The primal-dual regularization parameters $(\delta_w,\delta_c)$ are typically determined by examining the inertia of the system $K_2$.
In particular, the dual regularization $\delta_c$ is activated if zero eigenvalues are detected or if the filter line-search algorithm enters a feasibility restoration phase~\cite[Section 3.3]{wachter2006implementation}. The primal regularization $\delta_w$ is activated if the number of positive eigenvalues is less than $n + m_i$. For a positive $\delta_c$, the KKT system \eqref{eq:kkt:condensed} can be solved via the following equations:
\begin{equation}\label{eqn:primal}
  \left( K + \frac{1}{\delta_c} G^\top G \right) d_x = \bar{r}_1 - \frac{1}{\delta_c} \bar{r}_2 \;,
\end{equation}
and by setting $d_y = \delta_c^{-1} (G d_x - \bar{r}_2)$. The system in
\eqref{eqn:primal} is a Schur complement of the original condensed system in
\eqref{eq:kkt:condensed} and is positive definite if the system in
\eqref{eq:kkt:condensed} has the correct inertia $(n, m_e, 0)$.
Under standard inertia correction, it is
always Cholesky-factorizable.
Thus, when discussing the special treatment of the condensed KKT system
\eqref{eq:kkt:condensed}, we assume that the dual regularization is $\delta_c
= 0$, as $\delta_c > 0$ simplifies the solution procedure.
\end{remark}

\subsection{Golub \& Greif strategy: HyKKT}
\label{sec:kkt:golubgreif}
The Golub \& Greif~\cite{golub2003solving} strategy reformulates the KKT
system using an Augmented Lagrangian formulation.
It has been recently revisited in \cite{regev2023hykkt} to solve the
condensed KKT system~\eqref{eq:kkt:condensed} on GPUs.
The key idea is to reformulate the condensed KKT system \eqref{eq:kkt:condensed} in an equivalent form:
\begin{equation}
  \label{eq:kkt:hykkt}
  \begin{bmatrix}
    K_\gamma & G^\top \\
    G & 0
  \end{bmatrix}
  \begin{bmatrix}
    d_x \\ d_y
  \end{bmatrix}
  =
  \begin{bmatrix}
    \bar{r}_1 + \gamma G^\top \bar{r}_2 \\
    \bar{r}_2
  \end{bmatrix} \; ,
\end{equation}
where $K_\gamma := K + \gamma G^\top G$ is the regularized matrix (assuming $\delta_c=0$ per Remark \ref{rmk:dualreg}).
Let $\begin{bmatrix} Y & Z \end{bmatrix}$ be an orthogonal
matrix, where the columns of $Z$ form a basis for the null space of the Jacobian $G$.
Using a classical result from~\cite{debreu1952definite}, if $G$ has full row
rank, then there exists a threshold value $\underline{\gamma}$ such that for
all $\gamma > \underline{\gamma}$, the reduced Hessian $Z^\top K Z$ is
positive definite if and only if $K_\gamma$ is positive definite.
Using Sylvester's law of inertia stated in \eqref{eq:ipm:inertia}, we deduce that for $\gamma > \underline{\gamma}$, $\inertia(K_2) = (n + m_i, m_e + m_i, 0)$ if and only if the condensed matrix $K_\gamma$ is positive definite.

The linear solver HyKKT~\cite{regev2023hykkt} leverages the positive definiteness of $K_\gamma$ to solve \eqref{eq:kkt:hykkt} using a hybrid direct-iterative method with the following steps:
\begin{enumerate}
  \item Assemble $K_\gamma$ and factorize it using sparse Cholesky;
  \item Solve the Schur complement of \eqref{eq:kkt:hykkt} using a conjugate gradient (\CG) algorithm to recover the dual descent direction:
    \begin{equation}
      \label{eq:kkt:schurcomplhykkt}
      (G K_\gamma^{-1} G^\top) d_y = G K_\gamma^{-1} (\bar{r}_1 + \gamma G^\top \bar{r}_2) - \bar{r}_2 \; ;
    \end{equation}
  \item Solve the system $K_\gamma d_x = \bar{r}_1 + \gamma G^\top \bar{r}_2 - G^\top d_y$ to recover the primal descent direction.
\end{enumerate}
The method employs a sparse Cholesky factorization along with the conjugate gradient (\CG) algorithm \cite{hestenes-stiefel-1952}.
Sparse Cholesky factorization is advantageous because it is stable without
numerical pivoting, making the algorithm tractable on GPUs.
Each \CG iteration requires applying sparse triangular solves with the factors of $K_\gamma$.
For this reason, HyKKT is efficient only if the \CG solver converges in few iterations.
Fortunately, the eigenvalues of the Schur complement $S_\gamma := G
K_\gamma^{-1} G^\top$ all converge to $\frac{1}{\gamma}$ as the
regularization parameter $\gamma$ increases \cite[Theorem 4]{regev2023hykkt},
implying that $\lim_{\gamma \to \infty} \cond(S_\gamma) = 1$.
Since the convergence of the \CG method depends on the number of distinct
eigenvalues of $S_{\gamma}$, larger values of $\gamma$ lead to faster convergence of the \CG algorithm in \eqref{eq:kkt:schurcomplhykkt}.
Although we observe similar performance, these methods ensure a monotonic decrease in the residual norm of \eqref{eq:kkt:schurcomplhykkt} at each iteration.

\begin{remark}
    The parameter $\gamma$ should be set above a given threshold $\underline{\gamma}$ to ensure positive definiteness for $K_\gamma$.
    However, from a numerical standpoint, it is preferable to choose moderate
    values of $\gamma$, since large $\gamma$ increases the ill-conditioning of $K_\gamma$, thereby deteriorating the solution's quality.
    In \cite{golub2003solving}, the authors recommend choosing $\gamma = \| K \| / \|G\|^2$, which offers a good empirical trade-off.
    In \cite[Lemma 2.10]{gould2010spectral}, a tighter estimate is provided
    for the threshold $\underline{\gamma}$ if $G$ has full row rank and $Z^\top K Z$ is positive definite:
    \begin{equation}
      \label{eq:kkt:boundgamma}
      \underline{\gamma} = \dfrac{1}{\sigma_{\min}^2} \left(
        \dfrac{\|K\|^2}{\lambda_{\min}^Z} - \lambda_{\min}^Y
      \right) \; ,
    \end{equation}
    where $\sigma_{\min}$ is the smallest nonzero singular value of the
    Jacobian $G$, $\lambda_{\min}^Z > 0$ is the smallest eigenvalue of
    $Z^\top K Z$, and $\lambda_{\min}^Y$ the smallest eigenvalue of $Y^\top K Y$.
    From \eqref{eq:kkt:boundgamma}, we observe that $\underline{\gamma}$ is
    inversely proportional to $\lambda_{\min}^Z$ and $\sigma_{\min}$,
    indicating that it can reach a large magnitude if $G$ has nearly linearly dependent
    rows or if $Z^\top K Z$ is close to being indefinite.
    In other words, HyKKT may have difficulties solving problems with
    nearly dependent equality constraints or when the reduced Hessian is nearly indefinite.
\end{remark}

\subsection{Equality Relaxation Strategy: LiftedKKT}
\label{sec:kkt:sckkt}

For a small relaxation parameter $\tau > 0$ (chosen based on the numerical tolerance of the optimization solver $\varepsilon_{tol}$), the equality relaxation strategy~\cite{shin2023accelerating} approximates the equalities with lifted inequalities:
\begin{equation}
  \label{eq:problemrelaxation}
    \min_{x \in \mathbb{R}^n} \;  f(x)
\quad \text{subject to}\quad
     - \tau \leq g(x) \leq \tau \;,~  h(x) \leq 0  \; .
\end{equation}
The problem~\eqref{eq:problemrelaxation} has only inequality constraints. As
a result, the condensed KKT system \eqref{eq:kkt:condensed} reduces to
\begin{equation}
  \label{eq:liftedkkt}
    K_\tau \,d_x = - r_1 - H_\tau^\top(D_H r_4 - C r_2) \; ,
\end{equation}
where $H_\tau = \big(G^\top ~ H^\top\big)^\top$ and
$K_\tau := W + \delta_w I + H_\tau^\top D_H H_\tau$ (assuming $\delta_c=0$ per Remark \ref{rmk:dualreg}). Using the relation~\eqref{eq:ipm:inertia}, the matrix $K_\tau$
is guaranteed to be positive definite if the primal regularization parameter
$\delta_w$ is adequately large. As such, $\delta_w$ is chosen dynamically
using the inertia information of the system in \eqref{eq:kkt:condensed}.
Therefore, $K_\tau$ can be factorized using Cholesky decomposition,
satisfying the requirement for stable pivoting on GPUs.
The relaxation introduces an error in the final solution. Fortunately,
this error is of the same order as the solver tolerance, so it does not
significantly degrade the solution quality for small $\varepsilon_{tol}$.

While this method can be implemented with minor modifications to the
optimization solver, the presence of tight inequalities in
\eqref{eq:problemrelaxation} causes severe ill-conditioning throughout the
IPM iterations. Therefore, an accurate iterative refinement algorithm
is necessary to achieve reliable convergence behavior.

\subsection{Discussion}
We have introduced two algorithms to solve KKT systems on GPUs.
Unlike classical implementations, these methods do not require computing a
sparse \lblt factorization of the KKT system. Instead, they utilize alternative
reformulations based on the condensed KKT system~\eqref{eq:kkt:condensed}.
Both strategies rely on Cholesky factorization: HyKKT factorizes a
positive-definite matrix $K_\gamma$ obtained using an Augmented Lagrangian strategy,
whereas LiftedKKT factorizes a positive-definite matrix $K_\tau$ after employing an equality relaxation strategy.
We will demonstrate in the next section that the ill-conditioned matrices
$K_\gamma$ and $K_\tau$ possess a specific structure that limits accuracy loss in the IPM.

\section{Conditioning of the condensed KKT system}
The condensed matrix $K$ appearing in \eqref{eq:kkt:condensed} is known to become
increasingly ill-conditioned as the primal-dual iterates approach a local solution with active inequalities.
This behavior is amplified for the matrices $K_\gamma$ and $K_\tau$,
since HyKKT and LiftedKKT require
large $\gamma$ and small $\tau$, respectively. In this section, we analyze the
numerical error associated with the solution of the condensed KKT
system and discuss how the structured ill-conditioning renders its
adverse effect relatively benign.

We first discuss the perturbation bound for a generic linear system $Mx = b$.
The relative error after perturbing the right-hand side by $\Delta b$ is bounded by:
\begin{subequations}
  \label{eq:cond:defaultbond}
\begin{equation}
  \| \Delta x \| \leq \| M^{-1} \| \| \Delta b \| \;, \quad
  \frac{\| \Delta x \| }{\| x \|} \leq \cond(M) \, \frac{\| \Delta b \|}{\|b \|} \;.
\end{equation}
If the matrix is perturbed by $\Delta M$, the perturbed solution
$\widehat{x}$ satisfies $\Delta x = \widehat{x}- x =  - (M + \Delta M)^{-1} \Delta M \widehat{x}$.
If $\cond(M) \approx \cond(M + \Delta M)$, we have $M \Delta x \approx -\Delta M x$ (neglecting second-order terms),
giving the bounds
\begin{equation}
  \| \Delta x \| \lessapprox \|M^{-1}\| \|\Delta M \| \|x \| \; , \quad
  \frac{\| \Delta x \|}{\|x\|} \lessapprox \cond(M)\frac{\|\Delta M \|}{\|M\|} \; .
\end{equation}
\end{subequations}
The relative errors are bounded above by a term depending on the conditioning $\cond(M)$.
Hence, it is legitimate to investigate the impact of ill-conditioning
when solving the condensed system~\eqref{eq:kkt:condensed} with LiftedKKT or HyKKT.
We will show that we can tighten the bounds in \eqref{eq:cond:defaultbond}
by exploiting the structured ill-conditioning of the condensed matrix $K$.
We base our analysis on \cite{wright1998ill}, where
the author emphasizes the condensed KKT
system~\eqref{eq:kkt:condensed} without equality constraints. We generalize these results to the
matrix $K_\gamma$, which incorporates both equality and inequality
constraints. The results extend directly to $K_\tau$ (by setting the number of equalities to be zero).

\subsection{Centrality conditions}
We begin by recalling important results regarding the iterates of the interior-point algorithm.
For $p := (x, s, y, z)$, we denote by
$(p, v)$ the current primal-dual iterate,
and $(p^\star, v^\star)$ a solution of the KKT conditions~\eqref{eq:kktconditions}.
We denote $\cactive = \cactive(x^\star)$ the active-set at the optimal solution $x^\star$,
and $\cinactive = \cinactive(x^\star)$ the inactive set.
In this section, we focus on the \emph{local} convergence behavior of the
primal-dual iterate, assuming $(p, v)$ is sufficiently close to the solution
$(p^\star, v^\star)$.

\begin{assumption}
  \label{hyp:ipm}
  Let $(p^\star, v^\star)$ be a primal-dual solution
  satisfying the KKT conditions~\eqref{eq:kktconditions}. Assume the following:
  \begin{itemize}
  \item Continuity: The Hessian $\nabla^2_{x x} L(\cdot)$ is Lipschitz continuous
    near $p^\star$;
  \item Linear Independence Constraint Qualification (LICQ): The active Jacobian $A(x^\star)$ has full row rank;
  \item Strict Complementarity (SCS): For every $i \in \mathcal{B}(x^\star)$, $z_i^\star > 0$.
  \item Second-order sufficiency (SOSC): For every $h \in \nullspace\!\big(A(x^\star)\big)$,
    $h^\top \nabla_{x x}^2 L(p^\star)h > 0$.
  \end{itemize}
\end{assumption}

We denote $\delta(p, v) = \| (p, v) - (p^\star, v^\star) \|$ the Euclidean distance to the
primal-dual stationary point $(p^\star, v^\star)$.
From \cite[Theorem 2.2]{wright2001effects}, if Assumption~\ref{hyp:ipm}
holds at $p^\star$ and $v > 0$,
\begin{equation}
  \delta(p, v) = \Theta\left( \left\Vert \begin{bmatrix}
      \nabla_p L(p, v) \\ \min(v, s)
  \end{bmatrix}
  \right\Vert \right) \; .
\end{equation}
For feasible iterate $(s, v) > 0$, we define the \emph{duality measure} $\Xi(s, v)$ as the mapping
\begin{equation}
  \Xi(s, v) = s^\top v / m_i \; , 
\end{equation}
where $m_i$ is the number of inequality constraints.
The duality measure quantifies the satisfaction of the complementarity
constraints. For a solution $(p^\star, v^\star)$, $\Xi(s^\star, v^\star) = 0$.
The duality measure can be used to define the barrier parameter in IPM.

We assume the iterates $(p, v)$ satisfy the \emph{centrality conditions}
\begin{subequations}
  \label{eq:centralitycond}
  \begin{align}
    & \| \nabla_p \mathcal{L}(p, v) \| \leq C \; \Xi(s, v) \;,  \\
    \label{eq:centralitycond:complement}
    & (s, v) > 0 \;,\quad s_i v_i \geq \alpha \, \Xi(s, v) \quad \forall i =1, \cdots, m_i \; ,
  \end{align}
\end{subequations}
for some constants $C > 0$ and $\alpha \in (0, 1)$.
Conditions~\eqref{eq:centralitycond:complement} ensure that the products
$s_i v_i$ are not too disparate in the diagonal term $D_s$.
This condition is satisfied (albeit rather loosely)
in the Ipopt solver (see \cite[Equation (16)]{wachter2006implementation}).

\begin{proposition}[\cite{wright2001effects}, Lemma 3.2 and Theorem 3.3]
  \label{prop:cond:boundslack}
  Suppose $p^\star$ satisfies Assumption~\ref{hyp:ipm}.
  If the current primal-dual iterate $(p, v)$ satisfies the centrality
  conditions~\eqref{eq:centralitycond}, then
  \begin{subequations}
    \begin{align}
      i \in \mathcal{B} \implies s_i = \Theta(\Xi) \, , \quad v_i = \Theta(1) \;, \\
      i \in \mathcal{N} \implies s_i = \Theta(1) \, , \quad v_i = \Theta(\Xi) \; .
    \end{align}
    and the distance to the solution $\delta(p, v)$ is bounded by the duality measure $\Xi$:
    \begin{equation}
      \delta(p, v) = O(\Xi) \; .
    \end{equation}
  \end{subequations}
\end{proposition}

\subsection{Structured ill-conditioning}
This subsection examines the structure
of the condensed matrix $K_\gamma$ in HyKKT. All results
apply directly to the matrix $K_\tau$ in LiftedKKT, by setting the number
of equality constraints to $m_e = 0$.
First, we show that if the iterates $(p, v)$ satisfy
the centrality conditions~\eqref{eq:centralitycond}, then the
condensed matrix $K_\gamma$ exhibits structured ill-conditioning.

\subsubsection{Invariant subspaces in $K_\gamma$}
Without regularization we have $K_\gamma = W + H^\top D_s H + \gamma G^\top G$, with
the diagonal $D_s = S^{-1} V$.
We denote by $m_a$ the cardinality of the active set $\mathcal{B}$,
$H_{\cactive}$ the Jacobian of active inequality constraints, $H_{\cinactive}$ the
Jacobian of inactive inequality constraints and by
$A := \begin{bmatrix} H_{\cactive}^\top & G^\top \end{bmatrix}^\top$ the active Jacobian.
We define the minimum and maximum active slack values as
\begin{equation}
  s_{min} = \min_{i \in \cactive} s_i \; , \quad
  s_{max} = \max_{i \in \cactive} s_i \; .
\end{equation}
We recall that $m_e$ is the number of equality constraints,
and define $\ell := m_e + m_a$.

We express the structured ill-conditioning of $K_\gamma$ by
modifying the approach outlined in \cite[Theorem 3.2]{wright1998ill} to account for the additional
term $\gamma G^\top G$ arising from the equality constraints.
We show that the matrix $K_\gamma$ has two invariant subspaces
(in the sense of \cite[Chapter 5]{stewart1990matrix}),
associated respectively with the range of the transposed active Jacobian
(\emph{large space}) and with the null space of the active Jacobian (\emph{small space}).

\begin{theorem}[Properties of $K_\gamma$]
  \label{thm:cond}
  Suppose the condensed matrix is evaluated at a primal-dual
  point $(p, \nu)$ satisfying~\eqref{eq:centralitycond},
  for sufficiently small $\Xi$.
  Let $\lambda_1, \cdots, \lambda_n$ be the $n$ eigenvalues of
  $K_\gamma$, ordered such that $|\lambda_1| \geq  \cdots \geq |\lambda_n|$.
  Let $\begin{bmatrix} Y & Z \end{bmatrix}$ be an orthogonal
  matrix, where the columns of $Z$ form a basis for the null-space of
  $A$. Let $\underline{\sigma} :=\min\left(\frac{1}{\Xi}, \gamma\right)$
  and $\overline{\sigma} := \max\left(\frac{1}{s_{min}}, \gamma\right)$.
  Then,
  \begin{enumerate}
    \item[(i)] The $\ell$ largest-magnitude eigenvalues of $K_\gamma$ are positive,
      with $\lambda_1 = \Theta(\overline{\sigma})$ and $\lambda_{\ell} = \Omega(\underline{\sigma})$.
    \item[(ii)] The $n-\ell$ smallest-magnitude eigenvalues of $K_\gamma$
      are $\Theta(1)$.
    \item[(iii)] If $0 < \ell < n$, then $\cond(K_\gamma) = \Theta(\overline{\sigma})$.
    \item[(iv)] There exist orthonormal matrices $\widetilde{Y}$ and $\widetilde{Z}$ for
      simple invariant subspaces of $K_\gamma$ such that $Y - \widetilde{Y} = O(\underline{\sigma}^{-1})$
      and $Z - \widetilde{Z} = O(\underline{\sigma}^{-1})$.
  \end{enumerate}
\end{theorem}
\begin{proof}
  We start the proof by separating the inactive constraints from the active constraints in $K_\gamma$:
  \begin{equation}
    K_\gamma = W + H_{\cinactive}^\top S_{\cinactive}^{-1} V_{\cinactive} H_{\cinactive}
    + A^\top D_\gamma A \, ,
    \quad
    \text{with} \quad D_\gamma = \begin{bmatrix} S_{\cactive}^{-1} V_{\cactive} & 0 \\ 0 & \gamma I \end{bmatrix} \; .
  \end{equation}
  By Assumption~\ref{hyp:ipm}, Lipschitz
  continuity implies that the Hessian and the inactive Jacobian
  are bounded: $W = O(1)$, $H_{\cinactive} = O(1)$.
  Proposition~\ref{prop:cond:boundslack} implies
  $s_{\cinactive} = \Theta(1)$ and $v_{\cinactive} = \Theta(\Xi)$. We deduce:
  \begin{equation}
    \label{eq:cond:inactiveblock}
    H_{\cinactive}^\top S_{\cinactive}^{-1} V_{\cinactive} H_{\cinactive} = O(\Xi) \; .
  \end{equation}
  For sufficiently small $\Xi$,
  the condensed matrix $K_\gamma$ is dominated by the block of active constraints:
  \begin{equation}
    K_\gamma = A^\top D_\gamma A + O(1) \; .
  \end{equation}
  Sufficiently close to the optimum $p^\star$, the constraints qualification
  in Assumption~\ref{hyp:ipm} implies that $A = \Theta(1)$ and has rank $\ell$.
  The eigenvalues $\{\eta_i\}_{i =1,\cdots,n}$ of $A^\top D_\gamma A$
  satisfy $\eta_i > 0$ for $i = 1,\cdots,\ell$ and $\eta_i = 0$ for $i = \ell+1, \cdots, n$.
  As $s_{\cactive} = \Theta(\Xi)$ and $v_{\cactive} = \Theta(1)$
  (Proposition~\ref{prop:cond:boundslack}), the smallest diagonal
  element in $D_\gamma$ is $\Omega(\min\{\frac{1}{\Xi}, \gamma\})$
  and the largest diagonal element is $\Theta(\max\{\frac{1}{s_{min}}, \gamma\})$.
  Hence,
  \begin{equation}
    \eta_1 = \Theta(\overline{\sigma}) \; , \quad
    \eta_\ell = \Omega(\underline{\sigma}) \; .
  \end{equation}
  Using \cite[Lemma 3.1]{wright1998ill}, we deduce $\lambda_1 = \Theta(\overline{\sigma})$
  and $\lambda_\ell = \Omega(\underline{\sigma})$, proving the first result (i).

  Let $L_\gamma := A^\top D_\gamma A$.
  We have:
  \begin{equation}
    \begin{bmatrix}
      Z^\top \\ Y^\top
    \end{bmatrix}
    L_\gamma \begin{bmatrix}Z & Y \end{bmatrix}
    = \begin{bmatrix}
      L_1 & 0 \\
      0 & L_2
    \end{bmatrix} \; ,
  \end{equation}
  where $L_1 = 0$ and $L_2 = Y^\top L_\gamma Y$.
  The smallest eigenvalue of $L_2$ is $\Omega(\underline{\sigma})$
  and the matrix $E := K_\gamma - L_\gamma$ is $O(1)$.
  By applying \cite[Theorem 3.1, (ii)]{wright1998ill},
  the $n - \ell$ smallest eigenvalues in $K_\gamma$ differ by
  $\Omega(\underline{\sigma}^{-1})$ from those of the reduced Hessian $Z^\top K_\gamma Z$.
  In addition, \eqref{eq:cond:inactiveblock} implies
  that $Z^\top K_\gamma Z - Z^\top W Z = O(\Xi)$. Using SOSC,
  $Z^\top W Z$ is positive definite for small enough $\Xi$, implying
  all its eigenvalues are $\Theta(1)$. Using again \cite[Lemma 3.1]{wright1998ill},
  we get that the $n-\ell$ smallest eigenvalues in $K_\gamma$ are $\Theta(1)$,
  proving (ii). The results in (iii) can be obtained by combining
  (i) and (ii) (provided $0 < \ell < n$).
  Finally, point (iv) directly follows from \cite[Theorem 3.1 (i)]{wright1998ill}.
\end{proof}

\begin{corollary}
  \label{corr:cond:illstructured}
  The condensed matrix $K_\gamma$ can be decomposed as
  \begin{equation}
    \label{eq:cond:svd}
    K_\gamma = U \Sigma U^\top = \begin{bmatrix} U_L & U_S \end{bmatrix}
    \begin{bmatrix}
      \Sigma_L & 0 \\ 0 & \Sigma_S
    \end{bmatrix}
    \begin{bmatrix}
      U_L^\top \\ U_S^\top
    \end{bmatrix}
    \; ,
  \end{equation}
  where $\Sigma_L = \diag(\sigma_1, \cdots, \sigma_\ell) \in \mathbb{R}^{\ell \times \ell}$ and $\Sigma_S  = \diag(\sigma_{\ell+1}, \cdots, \sigma_n)\in \mathbb{R}^{(n-\ell) \times (n-\ell)}$
  two diagonal matrices, and $U_L \in \mathbb{R}^{n \times \ell}$,
  $U_S \in \mathbb{R}^{n \times (n - \ell)}$ two
  orthogonal matrices such that $U_L^\top U_S = 0$.
  The diagonal elements in $\Sigma_S$ and $\Sigma_L$ satisfy
  \begin{equation}
    \label{eq:cond:svddiag}
    \frac{\sigma_1}{\sigma_\ell} \ll \frac{\sigma_1}{\sigma_n} \; , \quad
    \frac{\sigma_{\ell +1}}{\sigma_{n}} \ll \frac{\sigma_1}{\sigma_n} \; .
  \end{equation}
  For suitably chosen basis $Y$ and $Z$,
  spanning respectively the row space and the null space of
  the active Jacobian $A$, we get
  \begin{equation}
    \label{eq:cond:invariantsubpsace}
    U_L - Y = O(\underline{\sigma}^{-1}) \; , \quad
    U_S - Z = O(\underline{\sigma}^{-1}) \; .
  \end{equation}
\end{corollary}
\begin{proof}
  Using the spectral theorem, we obtain the decomposition as \eqref{eq:cond:svd}.
  According to Theorem~\ref{thm:cond}, the $\ell$ largest eigenvalues of $K_\gamma$ are large
  and well separated from the $n - \ell$ smallest eigenvalues,
  establishing \eqref{eq:cond:svddiag}.
  Using Theorem \ref{thm:cond}, part (iv), we obtain the result
  in \eqref{eq:cond:invariantsubpsace}.
\end{proof}
Corollary~\ref{corr:cond:illstructured} provides deeper insight into the structure
of the condensed matrix $K_\gamma$.
Using~\eqref{eq:cond:invariantsubpsace}, we can associate
the large space of $K_\gamma$ with $\rangespace(A^\top)$
and the small space with $\nullspace(A)$.
The decomposition~\eqref{eq:cond:svd} leads to the following relations
\begin{equation}
  \label{eq:cond:boundinvariantsubspace}
  \begin{aligned}
    & \| K_\gamma \| = \| \Sigma_L \| = \Theta(\overline{\sigma}) \; , &
    \Sigma_L^{-1} = O(\underline{\sigma}^{-1})  \;, \\
    & \| K_\gamma^{-1} \| = \| \Sigma_S^{-1} \| = \Theta(1) \, , &
  \Sigma_S = \Theta(1) \, .
  \end{aligned}
\end{equation}
The conditioning of $\Sigma_L$ depends on $\cond(A)$
and the ratio $\frac{s_{max}}{s_{min}} = O(\Xi \overline{\sigma})$.
The conditioning of $\Sigma_S$ reflects the conditioning of the reduced Hessian $Z^\top W Z$.

Three observations are noteworthy:
\begin{enumerate}
  \item Theorem~\ref{thm:cond} (iii) states us that $\cond(K_\gamma) = \Theta(\overline{\sigma})$,
    meaning that if $\gamma \geq \frac{1}{s_{min}}$,
    the conditioning $\cond(K_\gamma)$ increases linearly with $\gamma$,
    recovering a known result \cite{regev2023hykkt}.
  \item In early IPM iterations, the slacks are pushed away from the boundary
    and the number of active inequality constraints is $m_a = 0$. Ill-conditioning
    in $K_\gamma$ is caused only by $\gamma G^\top G$ and $\underline{\sigma} = \overline{\sigma} = \gamma$.
  \item In late IPM iterations, the active slacks converges to $0$.
    If $\frac{1}{\Xi} \leq \gamma \leq \frac{1}{s_{min}}$ the parameter $\gamma$
    does not increase the ill-conditioning of the condensed matrix $K_\gamma$.
\end{enumerate}

\subsubsection{Numerical accuracy of the condensed matrix $K_\gamma$}
In floating-point arithmetic, the condensed matrix $K_\gamma$ is evaluated as
\begin{multline*}
  \widehat{K}_\gamma = W + \Delta W + (A + \Delta A)^\top (D_\gamma + \Delta D_\gamma) (A + \Delta A) \\
  + (H_{\cinactive} + \Delta H_{\cinactive})^\top S_{\cinactive}^{-1} V_{\cinactive} (H_{\cinactive} + \Delta H_{\cinactive}) \; ,
\end{multline*}
where $\Delta W = O(\epstol)$, $\Delta H_{\cinactive}  = O(\epstol)$, $ \Delta A  = \Theta(\epstol)$,
$\Delta D_\gamma = O(\epstol \overline{\sigma})$: most
of the errors arise from the ill-conditioned diagonal terms in $D_\gamma$.
\begin{proposition}
  \label{prop:cond:boundcondensedmatrix}
  In floating-point arithmetic, the perturbation
  of the condensed matrix $K_\gamma$ satisfies
  $\Delta K_\gamma := \widehat{K_\gamma} - K_\gamma  = O(\epstol \overline{\sigma})$.
\end{proposition}
\begin{proof}
  As $A = \Theta(1)$, we have:
 $A^\top D_\gamma A = \Theta(\overline{\sigma})$ and
  $A^\top \Delta D_\gamma A = O(\epstol \overline{\sigma})$.
  Neglecting second-order terms, we get
  \begin{multline*}
    \Delta K_\gamma =
    \overbrace{\Delta W}^{O(\epstol)}
    + \overbrace{\Delta A^\top D_\gamma A}^{O(\overline{\sigma}\epstol)}
    + \overbrace{A^\top D_\gamma \Delta A}^{O(\overline{\sigma}\epstol)}
    + \overbrace{A^\top \Delta D_\gamma A}^{O(\overline{\sigma}\epstol)}  \\
    + \underbrace{\Delta H_{\cinactive} S_{\cinactive}^{-1} V_{\cinactive} H_{\cinactive}}_{O(\epstol)}
    + \underbrace{H_{\cinactive} S_{\cinactive}^{-1} V_{\cinactive} \Delta H_{\cinactive} }_{O(\epstol)}
      \; ,
  \end{multline*}
  where the terms in braces show the respective bounds on the errors.
  We deduce the error is dominated by the terms arising from the active Jacobian,
  all bounded by $O(\overline{\sigma} \epstol)$, hence concluding the proof.
\end{proof}

If sufficiently large, the unstructured perturbation $\Delta K_\gamma$
can impact the structured ill-conditioning of the perturbed matrix $\widehat{K}_\gamma$.
The smallest eigenvalue $\eta_\ell$ of $A^\top D_\gamma A$ is $\Omega(\underline{\sigma})$. As noted in
\cite[Section 3.4.2]{wright1998ill}, the perturbed matrix $\widehat{K}_\gamma$ retains the $p$ large eigenvalues
bounded below by $\underline{\sigma}$ if the perturbation is itself much smaller than the eigenvalue $\eta_\ell$:
\begin{equation}
  \label{eq:cond:perturbationbound}
  \| \Delta K_\gamma \| \ll \eta_\ell = \Omega(\underline{\sigma})  \; .
\end{equation}
The bound in Proposition~\ref{prop:cond:boundcondensedmatrix} is too loose
for \eqref{eq:cond:perturbationbound} to hold without further assumption
(we have only $\underline{\sigma} \leq \overline{\sigma}$).
We note that for some constant $C > 0$, $\Delta K_\gamma \leq C \epstol \overline{\sigma}$,
implying $\Delta K_\gamma / \underline{\sigma} \leq C \epstol \overline{\sigma} / \underline{\sigma}$.
Hence, if we suppose in addition the ratio $\overline{\sigma}/\underline{\sigma}$ is near $1$, then
$\| \Delta K_\gamma\|= O(\epstol \underline{\sigma})$, ensuring \eqref{eq:cond:perturbationbound} holds.

\subsubsection{Numerical solution of the condensed system}
We aim to estimate the relative error
in solving the system $K_\gamma x = b$ in floating
point arithmetic. We suppose $K_\gamma$ is factorized using
a backward-stable Cholesky decomposition. The computed
solution $\widehat{x}$ solves a perturbed system
$\widetilde{K}_\gamma \widehat{x} = b$, where $\widetilde{K}_\gamma
= K_\gamma + \Delta_s K_\gamma$ and $\Delta_s K_\gamma$ a symmetric matrix satisfying
\begin{equation}
  \label{eq:cond:backwardstable}
  \|\Delta_s K_\gamma\| \leq \epstol \varepsilon_n \|K_\gamma\| \;,
\end{equation}
for a small constant $\varepsilon_n$ depending on the dimension $n$.
We need the following additional assumptions to
ensure (a) the Cholesky factorization runs to completion
and (b) we can incorporate the backward-stable perturbation $\Delta_s K_\gamma$
in the generic perturbation $\Delta K_\gamma$ introduced in
Proposition~\ref{prop:cond:boundcondensedmatrix}.
\begin{assumption} Let $(p, v)$ be the current primal-dual iterate. We assume:
  \begin{itemize}
    \item[(a)] $(p, v)$ satisfies the centrality conditions~\eqref{eq:centralitycond}.
    \item[(b)] The parameter $\gamma$ satisfies $\gamma = \Theta(\Xi^{-1})$.
    \item[(c)] The duality measure is large enough relative to the precision $\epstol$: $\epstol \ll \Xi$.
    \item[(d)] The primal step $\widehat{x}$ is computed using a backward
      stable method satisfying~\eqref{eq:cond:backwardstable} for a small constant
      $\varepsilon_n$.
  \end{itemize}
  \label{hyp:cond:wellcond}
\end{assumption}
Condition (a) implies both
$s_{min} = \Theta(\Xi)$ and $s_{max} = \Theta(\Xi)$ (Proposition \ref{prop:cond:boundslack}).
Condition (b) supposes in addition $\gamma = \Theta(\Xi^{-1})$, making
the matrix $\Sigma_L$ well-conditioned with
$\underline{\sigma} = \Theta(\Xi^{-1})$,
$\overline{\sigma} = \Theta(\Xi^{-1})$ and $\overline{\sigma}/\underline{\sigma} = \Theta(1)$.
This is in line with the recommendation in \cite{golub2003solving},
  where the authors recommend chosing $\gamma = \frac{\|K\|}{\|G\|^2}$
  (as here the condition (a) implies that $\|K \| = \Theta(\Xi)$).
Condition (c) ensures that $\cond(K_\gamma) = \Theta(\overline{\sigma})$
satisfies $\cond(K_\gamma) \epstol \ll 1$
(implying the Cholesky factorization runs to completion).
Condition (d) tells us that the perturbation caused by the Cholesky
factorization is $\Delta_s K_\gamma = O(\epstol \| K_\gamma\|)$. As
\eqref{eq:cond:boundinvariantsubspace} implies $\|K_\gamma \| = \Theta(\Xi^{-1})$,
we can incorporate $\Delta_s K_\gamma$ in the perturbation
$\Delta K_\gamma$ given in Proposition~\ref{prop:cond:boundcondensedmatrix}.

We now analyze the perturbation bound for the condensed system.
Let $x$ be the solution of the linear system $K_\gamma x = b$
in exact arithmetic, and $\widehat{x}$ the solution of
the perturbed system $\widehat{K}_\gamma \widehat{x} = \widehat{b}$
in floating-point arithmetic. We bound
the error $\Delta x = \widehat{x} - x$. We
recall that every vector $x \in \mathbb{R}^n$ decomposes as
\begin{equation}
  x = U_L x_L + U_S x_S = Y x_Y + Z x_Z \; .
\end{equation}

\paragraph{Impact of right-hand-side perturbation.}
Using \eqref{eq:cond:svd}, the inverse of
$K_\gamma$ satisfies
\begin{equation}
  \label{eq:cond:inversecondensed}
  K_\gamma^{-1}  = \begin{bmatrix} U_L & U_S \end{bmatrix}
  \begin{bmatrix}
    \Sigma_L^{-1} & 0 \\ 0 & \Sigma_S^{-1}
  \end{bmatrix}
  \begin{bmatrix}
    U_L^\top \\ U_S^\top
  \end{bmatrix}
  \; .
\end{equation}
Solving the system for $\widehat{b} := b + \Delta b$,
$\Delta x = K_\gamma^{-1} \Delta b$ decomposes as
\begin{equation}
  \begin{bmatrix}
    \Delta x_L \\ \Delta x_S
  \end{bmatrix}
  =
  \begin{bmatrix}
    \Sigma_L^{-1} & 0 \\ 0 & \Sigma_S^{-1}
  \end{bmatrix}
  \begin{bmatrix}
    \Delta b_L \\ \Delta b_S
  \end{bmatrix}
  \; ,
\end{equation}
which in turn implies the following bounds:
\begin{equation}
  \label{eq:cond:rhserror}
     \| \Delta x_L \| \leq \| \Sigma_L^{-1} \| \| \Delta b_L \| \; ,\quad
    \| \Delta x_S \| \leq \| \Sigma_S^{-1} \| \| \Delta b_S \| \; .
\end{equation}
Since $\Sigma_L^{-1} = O(\Xi)$ and $\Sigma_S^{-1} = \Theta(1)$,
we deduce that the error $\Delta x_L$ is smaller by a factor
of $\Xi$ than the error $\Delta x_S$. The total error
$\Delta x = U_L \Delta x_L + U_S \Delta x_S$ is bounded by
\begin{equation}
  \label{eq:cond:rhserrorfull}
  \| \Delta x \|
  \leq  \| \Sigma_L^{-1} \| \| \Delta b_L \| + \| \Sigma_S^{-1} \| \| \Delta b_S \| =
  O(\|\Delta b \|) \; .
\end{equation}

\paragraph{Impact of matrix perturbation.}
Since $\|\Delta K_\gamma\| \ll \|K_\gamma\|$, we have that
\begin{equation}
  \label{eq:cond:invperturbed}
  \begin{aligned}
    (K_\gamma + \Delta K_\gamma)^{-1} &= (I + K_\gamma^{-1} \Delta K_\gamma)^{-1} K_\gamma^{-1} \; , \\
                                      &= K_\gamma^{-1} - K_\gamma^{-1}\Delta K_\gamma K_\gamma^{-1} + O(\|\Delta K_\gamma\|^2) \; .
  \end{aligned}
\end{equation}
Decompose $\Delta K_\gamma$ into
$\Gamma_L \in \mathbb{R}^{\ell \times n}$ and $\Gamma_S \in \mathbb{R}^{(n-\ell) \times n}$ such that
$\Delta K_\gamma = \begin{bmatrix}
  \Gamma_L \\ \Gamma _S
\end{bmatrix}$.
Using \eqref{eq:cond:inversecondensed} the first-order error is:
\begin{equation}
  \label{eq:cond:inversecondensederror}
  K_\gamma^{-1}\Delta K_\gamma K_\gamma^{-1}  =
U_L \Sigma_L^{-1} \Gamma_L \Sigma_L^{-1}U_L^\top  +
  U_S \Sigma_S^{-1} \Gamma_S \Sigma_S^{-1}U_S^\top  \;.
\end{equation}
Using \eqref{eq:cond:boundinvariantsubspace} and $(\Gamma_L, \Gamma_S)= O( \Xi^{-1}\epstol)$,
we obtain $\Sigma_L^{-1} \Gamma_L \Sigma_L^{-1} = O(\Xi \epstol)$
and $\Sigma_S^{-1} \Gamma_S \Sigma_S^{-1} = O(\Xi^{-1} \epstol)$.
We deduce that the error made in the large space is $O(\Xi\epstol)$ whereas
the error in the small space is $O(\Xi^{-1}\epstol )$.

\subsection{Solution of the condensed KKT system}
We use \eqref{eq:cond:rhserror} and \eqref{eq:cond:inversecondensederror}
to bound the error made when solving the condensed KKT system~\eqref{eq:kkt:condensed}
in floating-point arithmetic.
In all this section, we assume that
the primal-dual iterate $(p,v)$ satisfies Assumption~\ref{hyp:cond:wellcond}.
Using \cite[Corollary 3.3]{wright2001effects}, the solution $(d_x, d_y)$ of the
condensed KKT system \eqref{eq:kkt:condensed} in exact arithmetic satisfies
$(d_x, d_y) = O(\Xi)$.
In \eqref{eq:kkt:condensed}, the RHS $\bar{r}_1$ and $\bar{r}_2$
evaluate in floating-point arithmetic as
\begin{equation}
  \label{eq:cond:condensedrhs}
  \left\{
  \begin{aligned}
    \bar{r}_1 &= - \widehat{r}_1 + \widehat{H}^\top\big(\widehat{D}_{s} \widehat{r}_{4} - \widehat{r}_{2} \big) \;, \\
     \bar{r}_2 &= -\widehat{r}_3 \; .
  \end{aligned}
  \right.
\end{equation}
Using basic floating-point arithmetic, we get
$\widehat{r}_1 = r_1 + O(\epstol)$,
$\widehat{r}_3 = r_3 + O(\epstol)$,
$\widehat{r}_4 = r_4 + O(\epstol)$.
The error in the right-hand-side $r_2$ is impacted by the term $\mu S^{-1}e$:
under Assumption~\ref{hyp:cond:wellcond}, it impacts differently
the active and inactive components:
$\widehat{r}_{2,\cactive}= r_{2,\cactive} + O(\epstol)$ and
$\widehat{r}_{2,\cinactive}= r_{2,\cinactive} + O(\Xi \epstol)$.
Similarly, the diagonal matrix $\widehat{D}_s$ retains full accuracy only
w.r.t. the inactive components: $\widehat{D}_{s,\cactive} = D_{s,\cactive} + O(\Xi^{-1} \epstol)$
and $\widehat{D}_{s,\cinactive} = D_{s,\cinactive} + O(\Xi \epstol)$.

\subsubsection{Solution with HyKKT}
We analyze the accuracy achieved when we solve the condensed system~\eqref{eq:kkt:condensed}
using HyKKT,
and show that the error remains reasonable even for large values of
the regularization parameter $\gamma$.

\paragraph{Initial right-hand-side.}
Let $\widehat{s}_\gamma := \bar{r}_1 + \gamma \widehat{G}^\top \bar{r}_2$.
The initial right-hand side in~\eqref{eq:kkt:schurcomplhykkt}
is evaluated as
$\widehat{r}_\gamma :=\widehat{G} \widehat{K}_\gamma^{-1} \widehat{s}_\gamma - \bar{r}_2$.
The following proposition shows that despite an expression involving the inverse
of the ill-conditioned condensed matrix $K_\gamma$, the error made in $r_\gamma$
is bounded only by the machine precision $\epstol$.

\begin{proposition}
In floating point arithmetic, the error in the right-hand-side $\Delta \widehat{r}_\gamma$ satisfies:
\begin{equation}
  \label{eq:cond:errorrgamma}
  \Delta \widehat{r}_\gamma = -\Delta \bar{r}_2 + \widehat{G} \widehat{K}_\gamma^{-1} \Delta s_\gamma = O(\epstol) \;.
\end{equation}
\end{proposition}
\begin{proof}
Using \eqref{eq:cond:condensedrhs}, we have
\begin{equation*}
  \label{eq:cond:boundderivationhykkt}
  \begin{aligned}
  \bar{r}_1 + \gamma \widehat{G}^\top \bar{r}_2 &=
- \widehat{r}_1 + \gamma \widehat{G}^\top \widehat{r}_3+ \widehat{H}^\top\big(\widehat{D}_{s} \widehat{r}_{4} - \widehat{r}_{2} \big) \\
&=  -
\underbrace{\widehat{r}_1}_{O(\epstol)} +
\underbrace{
\widehat{H}_{\cinactive}^\top\big(\widehat{D}_{s,\cinactive} \widehat{r}_{4,\cinactive} - \widehat{r}_{2,\cinactive} \big)}_{O(\Xi \epstol)}
+ \underbrace{\widehat{A}^\top \begin{bmatrix}
  \widehat{D}_{s,\cactive} \widehat{r}_{4,\cactive} - \widehat{r}_{2,\cactive}  \\
  \gamma \widehat{r}_3
\end{bmatrix}}_{O(\Xi^{-1}\epstol)} \; .
  \end{aligned}
\end{equation*}
The error decomposes as $\Delta s_\gamma = Y \Delta s_Y  + Z \Delta s_Z
= U_L \Delta s_L + U_S \Delta s_S$.
We have $\Delta s_Y = O(\Xi^{-1} \epstol)$ and $\Delta s_Z = O(\epstol)$.
Using \eqref{eq:cond:invariantsubpsace}, we deduce
$\Delta s_L = U_L^\top \Delta s_\gamma = O(\Xi^{-1} \epstol)$ and
$\Delta s_S = U_S^\top \Delta s_\gamma = O(\epstol)$.
Using \eqref{eq:cond:boundinvariantsubspace} and \eqref{eq:cond:inversecondensed},
the error in the large space $\Delta s_L$ annihilates in the backsolve:
\begin{equation}
  \label{eq:cond:boundhykkt1}
  K_\gamma^{-1} \Delta s_\gamma = U_L \Sigma_L^{-1} \Delta s_L + U_S \Sigma_S^{-1} \Delta s_S  = O(\epstol)
  \; .
\end{equation}
Finally, using \eqref{eq:cond:invperturbed}, we get
\begin{equation}
  \widehat{G} \widehat{K}_\gamma^{-1} \Delta s_\gamma \approx
  \widehat{G} (I - K_\gamma^{-1}\Delta K_\gamma) K_\gamma^{-1} \Delta s_\gamma \; .
\end{equation}
Using \eqref{eq:cond:boundhykkt1}, the first term is $\widehat{G} K_\gamma^{-1} \Delta s_\gamma = O(\epstol)$.
We have in addition
\begin{equation}
  G K_\gamma^{-1}\Delta K_\gamma (K_\gamma^{-1} \Delta s_\gamma)  =
  \big[ G U_L \Sigma_L^{-1} \Gamma_L + G U_S \Sigma_S^{-1} \Gamma_S \big] (K_\gamma^{-1} \Delta s_\gamma) \; .
\end{equation}
Using again \eqref{eq:cond:invariantsubpsace}:
$G U_L = G Y + O(\Xi)$ and $G U_S = O(\Xi)$. Hence
$G U_L \Sigma_L^{-1} \Gamma_L = O(1)$ and $G U_S \Sigma_S^{-1} \Gamma_S = O(1)$.
Using \eqref{eq:cond:rhserrorfull}, we have $K_\gamma^{-1} \Delta G^\top = O(\epstol)$,
implying $\Delta G K_\gamma^{-1} \Delta K_\gamma (K_\gamma^{-1} \Delta s_\gamma) = O(\Xi^{-1} \epstol^{2})$.
Assumption~\ref{hyp:cond:wellcond} implies that $\Xi^{-1} \epstol^2 \ll \epstol$,
proving \eqref{eq:cond:errorrgamma}.
\end{proof}

\paragraph{Schur-complement operator.}
Solving~\eqref{eq:kkt:schurcomplhykkt}
involves the Schur complement $S_\gamma = G K_\gamma^{-1} G^\top$.
We show that the Schur complement
has a specific structure that limits accuracy loss
in the conjugate gradient algorithm.

\begin{proposition}
  Assume $(p, v)$ satisfies Assumption~\ref{hyp:cond:wellcond}.
  In exact arithmetic,
  \begin{equation}
    S_\gamma = GY \, \Sigma_L^{-1} \, Y^\top G^\top + O(\Xi^2) \; .
  \end{equation}
\end{proposition}
\begin{proof}
  Using \eqref{eq:cond:inversecondensed}, we have
  \begin{equation}
    G K_\gamma^{-1} G^\top =
    G U_L \Sigma_L^{-1} U_L^\top G^\top + G U_S \Sigma_S^{-1} U_S^\top G^\top \;.
  \end{equation}
  Using \eqref{eq:cond:invariantsubpsace}, we have $G U_L = GY + O(\Xi)$,
  and $G = O(1)$, implying
  \begin{equation}
    G U_L \Sigma_L^{-1} U_L^\top G^\top = G Y  \Sigma_L^{-1} Y^\top G^\top + O(\Xi^2) \; .
  \end{equation}
  Using again \eqref{eq:cond:invariantsubpsace}, we have $G U_S = O(\Xi)$.
  Hence, $G U_S \Sigma_S^{-1} U_S^\top G^\top = O(\Xi^2)$,
  concluding the proof.
\end{proof}
We adapt the previous proposition to bound the error made when evaluating
$\widehat{S}_\gamma$ in floating-point arithmetic.
\begin{proposition}
  Assume $(p, v)$ satisfies Assumption~\ref{hyp:cond:wellcond}.
  In floating-point arithmetic,
  \begin{equation}
    \label{eq:cond:errorSgamma}
    \widehat{S}_\gamma = S_\gamma + O(\epstol) \; .
  \end{equation}
\end{proposition}
\begin{proof}
  We denote $\widehat{G} = G + \Delta G$ (with $\Delta G = O(\epstol)$). Then
  \begin{equation}
    \begin{aligned}
      \widehat{S}_\gamma &= \widehat{G} \widehat{K}_\gamma^{-1} \widehat{G}^\top \; , \\
                         &\approx (G + \Delta G)\big(K_\gamma^{-1} - K_\gamma^{-1}\Delta K_\gamma K_\gamma^{-1}\big)(G + \Delta G)^\top \;, \\
                    &\approx S_\gamma - G \big(K_\gamma^{-1}\Delta K_\gamma K_\gamma^{-1} \big)G^\top
                    + K_\gamma^{-1} \Delta G^\top + \Delta G K_\gamma^{-1} \; .
    \end{aligned}
  \end{equation}
  The second line is given by \eqref{eq:cond:invperturbed},
  the third by neglecting the second-order errors.
  Using \eqref{eq:cond:rhserrorfull}, we get $K_\gamma^{-1} \Delta G^\top = O(\epstol)$
  and $\Delta G K_\gamma^{-1} = O(\epstol)$.
  Using \eqref{eq:cond:inversecondensederror}, we have
  \begin{equation*}
    G \big(K_\gamma^{-1}\Delta K_\gamma K_\gamma^{-1} \big)G^\top =
G U_L \Sigma_L^{-1} \Gamma_L \Sigma_L^{-1}U_L^\top G^\top  +
G U_S \Sigma_S^{-1} \Gamma_S \Sigma_S^{-1}U_S^\top  G^\top \;.
  \end{equation*}
  Using \eqref{eq:cond:invariantsubpsace}, we have $G U_S = O(\Xi)$.
  As $\Sigma_S^{-1} = \Theta(1)$ and $\Gamma_S = O(\Xi^{-1} \epstol)$, we
  get
  $G U_S \Sigma_S^{-1} \Gamma_S \Sigma_S^{-1}U_S^\top  G^\top = O(\Xi \epstol)$.
  Finally, as $\Sigma_L^{-1} = \Theta(\Xi)$ and $G U_L = GY + O(\Xi)$,
  we have
  \begin{equation}
    G U_L \Sigma_L^{-1} \Gamma_L \Sigma_L^{-1}U_L^\top G^\top =
    G Y \Sigma_L^{-1} \Gamma_L \Sigma_L^{-1}Y^\top G^\top + O(\Xi^2 \epstol) \; .
  \end{equation}
  We conclude the proof by using
  $G Y \Sigma_L^{-1} \Gamma_L \Sigma_L^{-1}Y^\top G^\top = O(\Xi \epstol)$.
\end{proof}
The two error bounds \eqref{eq:cond:errorrgamma} and
\eqref{eq:cond:errorSgamma} ensure that we can solve
\eqref{eq:kkt:schurcomplhykkt} using a conjugate gradient
algorithm, as the errors remain limited in floating-point
arithmetic.

\subsubsection{Solution with Lifted KKT system}
The equality relaxation strategy used in LiftedKKT
removes the equality constraints from the optimization problems, simplifying
the solution of the condensed KKT system to \eqref{eq:liftedkkt}.
The active Jacobian $A$ reduces to the active inequalities $A = H_{\cactive}$,
and we recover the original setting presented in \cite{wright1998ill}.
Using the same arguments as in \eqref{eq:cond:boundderivationhykkt},
the error in the right-hand-side is bounded by $O(\Xi^{-1} \epstol)$ and is in the
range space of the active Jacobian $A$. Using \eqref{eq:cond:inversecondensed},
we can show that the absolute error on $\widehat{d}_x$ is bounded by
$O(\Xi \epstol)$. That implies the descent direction $\widehat{d}_x$ retains
full relative precision close to optimality.
In other words, we can refine the solution returned by the Cholesky solver accurately using
Richardson iterations.

\subsubsection{Summary}
Numerically, the primal-dual step $(\widehat{d}_x, \widehat{d}_y)$
is computed only with an (absolute) precision $\varepsilon_{K}$,
greater than the machine precision $\epstol$ (for HyKKT, $\varepsilon_K$
is the absolute tolerance of the \CG algorithm, for LiftedKKT the
absolute tolerance of the iterative refinement algorithm).

The errors $\widehat{d}_x - d_x = O(\varepsilon_K)$ and
$\widehat{d}_y - d_y = O(\varepsilon_K)$ propagate further in $(\widehat{d}_s, \widehat{d}_z)$.
According to \eqref{eq:kkt:condensed}, we have $\widehat{d}_s = - \widehat{r}_4 - \widehat{H} \widehat{d}_x$.
By continuity, $\widehat{H} = H + O(\epstol)$ and $\widehat{r}_4 = r_4 + O(\epstol)$, implying
\begin{equation}
  \widehat{d}_s = d_s + O(\varepsilon_K) \; .
\end{equation}
Eventually, we obtain $\widehat{d}_z = - \widehat{r}_2 - \widehat{D}_s \widehat{d}_s$,
giving the following bounds for the errors in the inactive and active components:
\begin{equation}
  \begin{aligned}
     \widehat{d}_{z,\cactive} &= -\widehat{r}_{2,\cactive} - \widehat{D}_{s,\cactive} \widehat{d}_{s,\cactive}
    = d_{z,\cactive} + O(\varepsilon_K \Xi^{-1}) \;,\\
                              \widehat{d}_{z,\cinactive} &= -\widehat{r}_{2,\cinactive} - \widehat{D}_{s,\cinactive} \widehat{d}_{s,\cinactive}
                               = d_{z,\cinactive} + O(\varepsilon_K \Xi) \; .
  \end{aligned}
\end{equation}
Most of the error occurs in the descent direction for multipliers of the active bound $\widehat{d}_{z,\cactive}$.
The impact remains limited if there are few active inequalities.

\section{Implementation}
\label{sec:num:implementation}
All our implementations use the Julia language \cite{bezanson-edelman-karpinski-shah-2017}.
We utilized our local workstation to generate the results on the CPU, which is equipped with an AMD EPYC 7443 processor (3.2GHz) with 24 processors.
For the GPU results, we used an NVIDIA A100 GPU (with CUDA 12.3) on the Polaris testbed at Argonne National Laboratory
\footnote{\url{https://www.alcf.anl.gov/polaris}}.
\add{Julia uses Just-in-Time (JIT) compilation to compile the code at runtime.
  The wall-time reported in the benchmarks does not include the compilation time.
}

\paragraph{IPM solver.}
We have implemented the two condensed-space methods in our nonlinear IPM solver MadNLP~\cite{shin2021graph}.
This implementation utilizes the {\tt AbstractKKTSystem} abstraction
in MadNLP to represent various KKT linear systems.
MadNLP can execute most of the IPM algorithm on the GPU.
In particular, all array manipulations are performed by GPU kernels, avoiding data transfers to host memory.
We refer to \cite{shin2023accelerating} for a detailed description of the GPU implementation in MadNLP.

\paragraph{Evaluation of the nonlinear models.}
For fast evaluation of the derivatives on the GPU,
we use ExaModels~\cite{shin2023accelerating}.
ExaModels harnesses the sparsity structure and provides custom derivative
kernels for repetitive algebraic subexpressions in constraints and
objective functions, enabling parallel first and second-order derivative computations on the GPU~\cite{bischof1991exploiting,enzyme2021}.
This approach caters to the SIMD architecture of the GPU by assigning multiple threads
to compute derivatives for different expression values.
ExaModels is highly efficient at evaluating problem derivatives, offering speedups often greater than 100 compared
to standard tools such as AMPL or JuMP (see Figure~\ref{fig:examodels}). As a
result, the time spent evaluating the model becomes
negligible compared to the time spent in the linear solver.
ExaModels compiles the derivative evaluation kernels specifically for a
given problem, whereas JuMP and AMPL rely on a precompiled interpreter to evaluate the model's derivatives.
However, ExaModels requires compilation time to create the derivative
evaluation kernels, which can become substantial for problems involving a
large number of heterogeneous constraints.
Throughout this paper, reported timings exclude compilation time.

\begin{figure}[!ht]
  \centering
  \includegraphics[width=.9\textwidth]{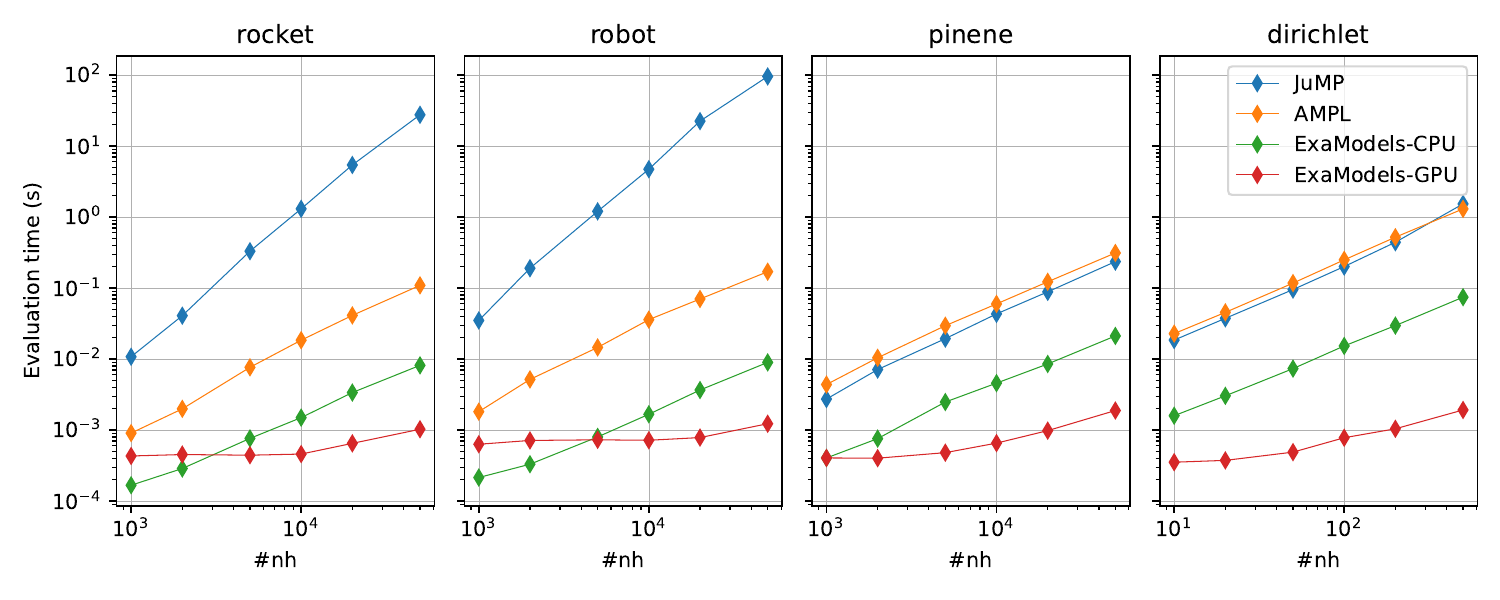}
  \caption{Time spent evaluating problem derivatives (gradient, Jacobian, Hessian)
      as the problem dimension {\tt nh} increases. This benchmark compares four
  instances from the COPS benchmark~\cite{dolan2002benchmarking}: {\tt rocket}, {\tt robot}, {\tt pinene}, and {\tt dirichlet}.}
  \label{fig:examodels}
\end{figure}

\paragraph{Linear solvers.}
The KKT systems assembled in MadNLP are solved using various sparse linear
solvers, selected based on the KKT formulation (\eqref{eq:kkt:condensed},
\eqref{eq:kkt:augmented}, \eqref{eq:kkt:unreduced}) and the device (CPU, GPU).
We use the following solvers:
\begin{itemize}
  \item {\tt HSL MA27}: Implements the multifrontal \lblt factorization on the CPU~\cite{duff1983multifrontal}.
    \add{HSL MA27 supports single-threading only.}
  \item {\tt HSL MA57}: Implements the multifrontal \lblt factorization on the CPU~\cite{duff1983multifrontal},
    using multithreaded BLAS kernels in the elimination algorithm.
  \item {\tt HSL MA86}: Implements a fine-grained multicore parallel supernodal \lblt factorization on the CPU.
  \item {\tt Panua Pardiso}: Implements shared-memory parallel supernodal Cholesky, \ldlt and \lblt factorizations on the CPU.
  \item {\tt CHOLMOD}: Implements Cholesky and \ldlt factorizations on the CPU 
    (using AMD ordering \cite{amestoy-david-duff-2004} by default).
    It factorizes the condensed matrices $K_\gamma$ and $K_\tau$ appearing
    in \eqref{eq:kkt:hykkt} and \eqref{eq:liftedkkt}, respectively.
    \add{CHOLMOD supports multithreading.}
  \item {\tt cuDSS}: Implements \llt, \ldlt and \lu decompositions on NVIDIA GPUs.
    We use \ldlt factorization to decompose the ill-conditioned condensed matrices $K_\gamma$ on the GPU,
    (here, we have observed that \ldlt is more robust than Cholesky factorization).
  \item {\tt Krylov.jl}: Contains the \CG method
    used in the Golub \& Greif strategy to solve \eqref{eq:kkt:schurcomplhykkt} on both CPU and GPU architectures.
    \add{In our experiments, Krylov.jl runs in serial on the CPU.}
\end{itemize}
CHOLMOD is included with Julia.
For HSL linear solvers, we use libHSL \cite{fowkes-lister-montoison-orban-2024} with the Julia interface HSL.jl \cite{montoison-orban-hsl-2021}.
HSL MA57 and CHOLMOD are both compiled with OpenBLAS, a multithreaded version of BLAS and LAPACK.
Krylov.jl~\cite{montoison2023krylov} provides a collection
of Krylov methods with a polymorphic implementation suitable for both
CPU and GPU architectures.

\section{Numerical results}
First, we assess in \S\ref{sec:num:pprof} the performance of the two hybrid solvers, LiftedKKT and HyKKT.
The GPU implementation is compared with state-of-the-art CPU-based solvers.
Then, we present results for the PGLIB OPF benchmark in \S\ref{sec:num:opf},
 complemented by the CUTEst benchmark in \S\ref{sec:num:cutest}.
\add{All results have been generated using double precision on both the CPU and GPU.}

\subsection{Performance analysis of HyKKT and LiftedKKT on a large-scale instance}
\label{sec:num:pprof}
We evaluate each KKT solver's performance on a large-scale OPF instance from
the PGLIB benchmark~\cite{babaeinejadsarookolaee2019power}: {\tt 78484epigrids}.
Our formulation with ExaModels comprises 674,562 variables, 661,017 equality constraints, and 378,045 inequality constraints.
Our previous work has pointed out that once the OPF model is evaluated
on the GPU using ExaModels, the time spent on automatic differentiation (AD) becomes negligible, leaving
the KKT solver as the primary bottleneck~\cite{shin2023accelerating}.

\subsubsection{Performance comparison of linear solvers on CPUs and GPUs}
We evaluate the performance of the {\tt cuDSS} solver when
factorizing the matrix $K_{\gamma}$ at the first IPM iteration (with
$\gamma = 10^7$). We compare the symbolic analysis,
factorization, and triangular solve times with those reported in CHOLMOD
(single-threaded), HSL MA86, and Panua Pardiso with sparse Cholesky.
The CPU implementations achieved maximum efficiency when configured with 8 threads;
therefore, we consistently present performance results
for parallel CPU solvers using 8 threads throughout this section.

The results are displayed in Table~\ref{tab:linsol:time}.
We benchmark the three decompositions implemented in {\tt cuDSS} (\llt, \ldlt, \lu).
We observe that the analysis phase in {\tt cuDSS} is four times slower than in CHOLMOD.
The ordering phase of symbolic analysis is performed on the CPU, and thus does not benefit from GPU parallelism \cite{nvidiaNVIDIACuDSSPreview}. Fortunately, the analysis can be performed only once, and the symbolic factorization can be reused if the optimization problem is re-solved with the same sparsity structure.
Numerical factorization in {\tt cuDSS} is approximately twice as fast as in HSL MA86, with timings almost independent of the factorization used.

\begin{table}[!ht]
  \centering
  \resizebox{.7\textwidth}{!}{
    \begin{tabular}{lrrr}
      \toprule
      & analysis (s)    & factorization (s) & backsolve (s) \\
      \midrule
      CHOLMOD        & $1.37\times 10^{0}$ & $7.44\times 10^{-1}    $ & $5.20\times 10^{-2}$ \\
      PARDISO        & $4.52\times 10^{0}$ & $2.94\times 10^{-1}    $ & $8.94\times 10^{-2}$ \\
      MA86           & $3.88\times 10^{0}$ & $1.38\times 10^{-1}    $ & $5.50\times 10^{-2}$ \\
      {\tt cuDSS}-cholesky & $3.63\times 10^{0}$ & $4.05\times 10^{-2}    $ & $7.22\times 10^{-3}$ \\
      {\tt cuDSS}-ldl      & $3.57\times 10^{0}$ & $4.21\times 10^{-2}    $ & $6.35\times 10^{-3}$ \\
      {\tt cuDSS}-lu       & $3.55\times 10^{0}$ & $6.28\times 10^{-2}    $ & $6.42\times 10^{-3}$ \\
      \bottomrule
    \end{tabular}
  }
  \caption{Performance comparison of {\tt cuDSS} with CHOLMOD, Pardiso, and HSL MA86
      (all running in parallel using 8 threads).
    The column for ``analysis'' shows the time spent on symbolic analysis (computation
    of a reordering to minimize fill-in and construction of the elimination tree).
    The columns for ``factorization'' and ``backsolves'' show the time spent on numerical
    factorization and the subsequent triangular solve, respectively.
    The matrix $K_\gamma$ is symmetric positive definite, with
    a size $n = 674,562$ and is extremely sparse, with only $7,342,680$ non-zero entries ($0.002$\%).
  }
  \label{tab:linsol:time}
\end{table}

\subsubsection{Tuning the Golub \& Greif strategy}
\label{sec:num:tuninghykkt}

We observe that as the regularization $\gamma$ increases, the \CG
algorithm converges faster. However, this comes at the cost of reduced accuracy.
Figure~\ref{fig:hybrid:gamma} depicts the evolution of the number of \CG
iterations and relative accuracy as $\gamma$ increases from $10^4$ to $10^8$ in HyKKT.
The number of \CG iterations decreases tenfold as $\gamma$ increases, but the relative
residual norm increases linearly with $\gamma$.
Thus, a trade-off exists between \CG execution time and solution accuracy.

The table in Figure~\ref{fig:hybrid:gamma} compares the IPM solution time on
the CPU (using Pardiso) and on the GPU (using {\tt cuDSS}).
Overall, {\tt cuDSS} is faster than Pardiso, leading to a 4x-8x speedup in
total IPM solution time. We also note that the assembly of the condensed
matrix $K_\gamma$ parallelizes effectively on the GPU, with assembly time reduced
from $\approx 8$s on the CPU to $\approx 0.2$s on the GPU
\add{(the assembly of the KKT system is not parallelized on the CPU)}.

\begin{figure}[!ht]
  \centering
  \resizebox{\textwidth}{!}{
  \begin{tabular}{|r|rrrr >{\bfseries}r|rrrr >{\bfseries}r|}
  \hline
  & \multicolumn{5}{c|}{\bf Pardiso (CPU)} & \multicolumn{5}{c|}{\bf cuDSS-\ldlt (CUDA)} \\
  \hline
  $\gamma$ & \# it & cond. (s) & \CG (s) & linsol (s) & IPM (s) & \# it & cond. (s) & \CG (s) & linsol (s) & IPM (s) \\
  \hline
  $10^4$ & 96 & 8.1 & 562.1 & 599.8 & 639.2 & 96 & 0.17 & 113.27 & 114.52 & 124.00 \\
  $10^5$ & 96 & 8.5 & 212.0 & 249.9 & 290.0 & 96 & 0.17 & 53.37  & 54.62  & 64.39  \\
  $10^6$ & 96 & 8.6 & 96.7  & 134.6 & 174.6 & 96 & 0.17 & 14.53  & 15.78  & 25.39  \\
  $10^7$ & 96 & 8.6 & 52.9  & 91.0  & 130.5 & 96 & 0.17 & 7.95   & 9.20   & 18.41  \\
  $10^8$ & 96 & 8.6 & 35.1  & 73.2  & 113.3 & 96 & 0.17 & 5.36   & 6.62   & 15.90  \\
  \hline
  \end{tabular}
  }
  \includegraphics[width=\textwidth]{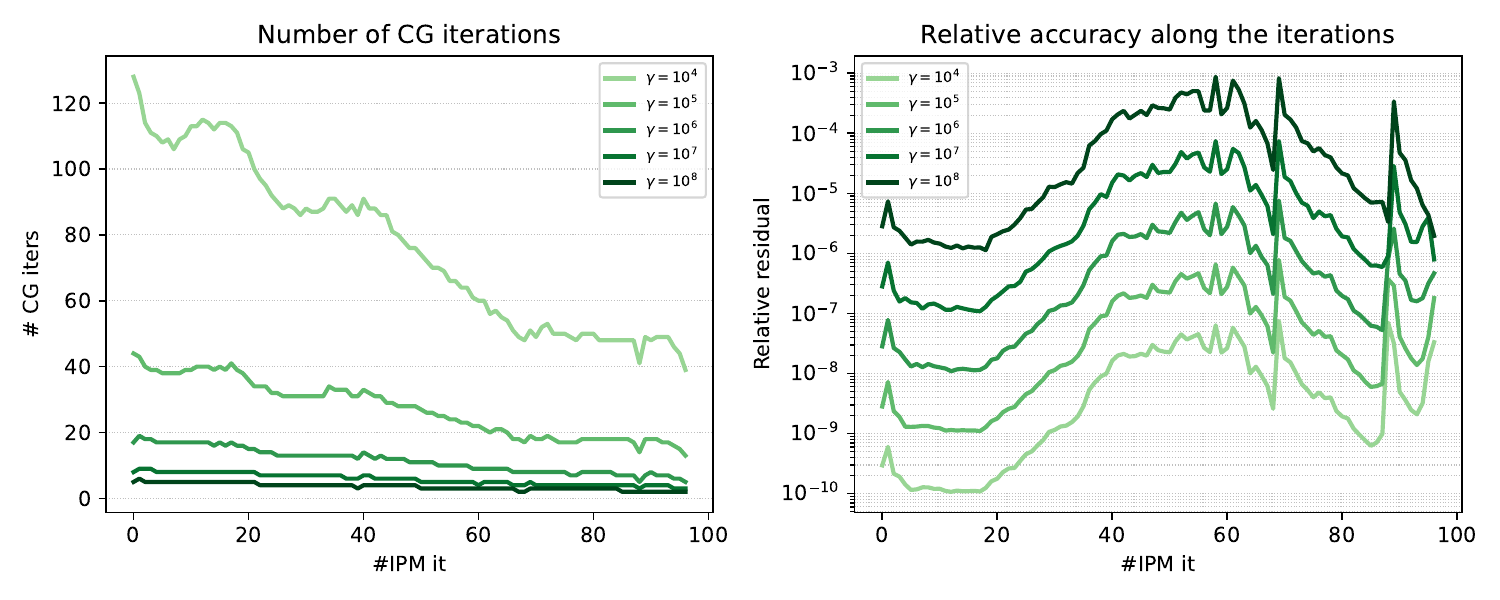}
  \caption{
    Above: Breakdown of IPM solution time across
    (a) condensation (cond.), (b) \CG (c) total linear solver
    (linsol.) and (d) total IPM solver time
    (IPM). Performance of {\tt cuDSS} is compared with Pardiso (8 threads on CPU).
    Below: Impact of $\gamma$ on the total number of \CG iterations
    and the relative residual norm at each IPM iteration.
    The peak in the relative residual norm corresponds
    to primal-dual regularizations within the IPM algorithm,
    applied when $K_\gamma$ is not positive definite.
    The {\tt it} column shows the number of iterations,
    while {\tt cond.}, {\tt CG}, {\tt linsol} and {\tt IPM}
    columns show the time spent on assembling the condensed KKT matrix,
    the time spent in the conjugate gradient, the time spent in the linear
    solver (including CG) and the IPM algorithm, respectively.
    \label{fig:hybrid:gamma}
  }
\end{figure}

\subsubsection{Tuning the equality relaxation strategy}
We now analyze the numerical performance of LiftedKKT (\S\ref{sec:kkt:sckkt}).
This method solves the KKT system~\eqref{eq:liftedkkt} using a direct solver.
The parameter $\tau$ used in the equality relaxation~\eqref{eq:problemrelaxation}
is set to the IPM tolerance $\varepsilon_{tol}$ (in practice, there is little benefit
to setting $\tau$ below $\varepsilon_{tol}$ as
inequality constraints are only satisfied to within $\pm \varepsilon_{tol}$ in the IPM).

We compare in Table~\ref{tab:sckkt:performance} the performance obtained by LiftedKKT
as we decrease the IPM tolerance $\varepsilon_{tol}$.
We present runtimes for both the CPU (Pardiso with Cholesky) and the GPU ({\tt cuDSS}-\ldlt).
Slacks associated with relaxed equality constraints converge to values below $2 \tau$,
leading to highly ill-conditioned terms in the diagonal matrices $D_s$.
Consequently, the conditioning of the matrix $K_\tau$ in \eqref{eq:liftedkkt} can exceed
$10^{18}$, leading to a nearly singular linear system.
However, the factorization succeeds and the loss of accuracy caused by the ill-conditioning is tamed by the multiple
Richardson iterations in the iterative refinement, maintaining relative
residual accuracy at an acceptable level.
As a result, {\tt cuDSS} solves
the problem to optimality in $\approx 20$s, a time comparable to HyKKT (see Figure~\ref{fig:hybrid:gamma}).
Updating $\tau$ changes the total number of IPM iterations,
as a different problem~\eqref{eq:problemrelaxation} is solved for each $\tau$ value.

\begin{table}[!ht]
  \centering
  \resizebox{.7\textwidth}{!}{
  \begin{tabular}{|l|rr|rr|r|}
    \hline
    & \multicolumn{2}{c|}{\bf Pardiso-\ldlt (CPU)} & \multicolumn{2}{c|}{\bf cuDSS-\ldlt (CUDA)}& \\
    \hline
    $\varepsilon_{tol}$ & \#it & time (s)&  \#it & time (s) & accuracy \\
    \hline
    $10^{-4}$ &127 & 164.9 & 114 & 19.9 & $1.2 \times 10^{-2}$ \\
    $10^{-5}$ &129 & 168.1 & 113 & 30.4 & $1.2 \times 10^{-3}$ \\
    $10^{-6}$ &200 & 301.7 & 109 & 25.0 & $1.2 \times 10^{-4}$ \\
    $10^{-7}$ &110 & 159.2 & 104 & 20.1 & $1.2 \times 10^{-5}$ \\
    $10^{-8}$ &111 & 156.9 & 105 & 20.3 & $1.2 \times 10^{-6}$ \\
    \hline
  \end{tabular}
  }
  \caption{Performance of the equality-relaxation
    strategy as the IPM tolerance $\varepsilon_{tol}$ decreases.
    Runtimes are shown for CPU (using Pardiso-\ldlt)
    and GPU ({\tt cuDSS}-\ldlt). \add{Accuracy is defined
    as the residual norm $\|K x - b\|_\infty$, where $x$ is the
    solution returned by LiftedKKT.}
  \label{tab:sckkt:performance}
  }
\end{table}

\subsubsection{Breakdown of the time spent in one IPM iteration}
We decompose the time spent in a single IPM iteration for LiftedKKT and HyKKT.
As a reference running on the CPU, we use HSL MA27.
The factorization time for HSL MA57 is 1.9s, compared to 1.2s for HSL MA27
(increasing the number of threads in the BLAS backend used by HSL MA57 yielded no improvement).
Thus, the block elimination algorithm in HSL MA57 is not beneficial here.

KKT solution time is divided into:
(1) assembling the KKT system,
(2) factorizing the KKT system, and (3) computing the descent direction via triangular solves.
As shown in Figure~\ref{fig:timebreakdown}, KKT construction represents only
a small fraction of the total time, compared to factorization and triangular solves.
LiftedKKT and HyKKT lead to faster factorization times
as the condensed KKT system can be efficiently factorized in parallel. HSL MA86
factorizes the condensed system nearly 10 times faster than HSL MA27, and
{\tt cuDSS} provides an additional 4x speed-up.
We focus on comparing parallel CPU solvers (MA86 and Pardiso) for Lifted and HyKKT,
as the factorization of the augmented KKT system~\eqref{eq:kkt:augmented}
does not parallelize well on CPUs~\cite{tasseff2019exploring}.
After factorization, LiftedKKT computes the descent direction faster than HyKKT
(0.01s versus 0.07s) because HyKKT must run the \CG algorithm to completion to solve the Schur complement
system~\eqref{eq:kkt:schurcomplhykkt}, requiring additional backsolves in the linear solver.

\begin{figure}[!ht]
  \centering
  \resizebox{.8\textwidth}{!}{
    \begin{tabular}{|ll|rrr|r|}
      \hline
      linear system & linear solver & build (s) & factorize (s) & backsolve (s) & total (s) \\
       \hline
      \eqref{eq:kkt:augmented} & HSL MA27          & $2.85 \times 10^{-2}$ & $1.23 \times 10^{0\phantom{-}}$  & $3.48 \times 10^{-1}$ & $1.61 \times 10^{0\phantom{-}}$ \\
      \hline
      LiftedKKT \eqref{eq:liftedkkt} & CHOLMOD & $8.55 \times 10^{-2}$ & $5.84\times10^{-1}$ & $1.17\times10^{-1}$ & $7.86 \times 10^{-1}$ \\
      LiftedKKT \eqref{eq:liftedkkt} & pardiso & $8.16 \times 10^{-2}$ & $2.23\times10^{-1}$ & $1.38\times10^{-1}$ & $4.43 \times 10^{-1}$ \\
      LiftedKKT \eqref{eq:liftedkkt} & MA86    & $8.57 \times 10^{-2}$ & $1.36\times10^{-1}$ & $1.18\times10^{-1}$ & $3.39 \times 10^{-1}$ \\
      LiftedKKT \eqref{eq:liftedkkt} & {\tt cuDSS}   & $2.26 \times 10^{-3}$ & $4.91\times10^{-2}$ & $1.09\times10^{-2}$ & $6.23 \times 10^{-2}$ \\
      \hline
      HyKKT \eqref{eq:kkt:hykkt} & CHOLMOD     & $8.01 \times 10^{-2}$ & $5.89\times10^{-1}$ & $7.22\times10^{-1}$ & $ 1.39 \times 10^{0\phantom{-}}$ \\
      HyKKT \eqref{eq:kkt:hykkt} & pardiso     & $9.06 \times 10^{-2}$ & $2.21\times10^{-1}$ & $9.19\times10^{-1}$ & $1.23 \times 10^{0\phantom{-}}$\\
      HyKKT \eqref{eq:kkt:hykkt} & MA86        & $8.60 \times 10^{-2}$ & $1.35\times10^{-1}$ & $7.50\times10^{-1}$ & $9.71 \times 10^{-1}$ \\
      HyKKT \eqref{eq:kkt:hykkt} & {\tt cuDSS}       & $1.98 \times 10^{-3}$ & $3.38\times10^{-2}$ & $6.66\times10^{-2}$ & $1.02 \times 10^{-1}$ \\
      \hline
    \end{tabular}
  }
  \includegraphics[width=.9\textwidth]{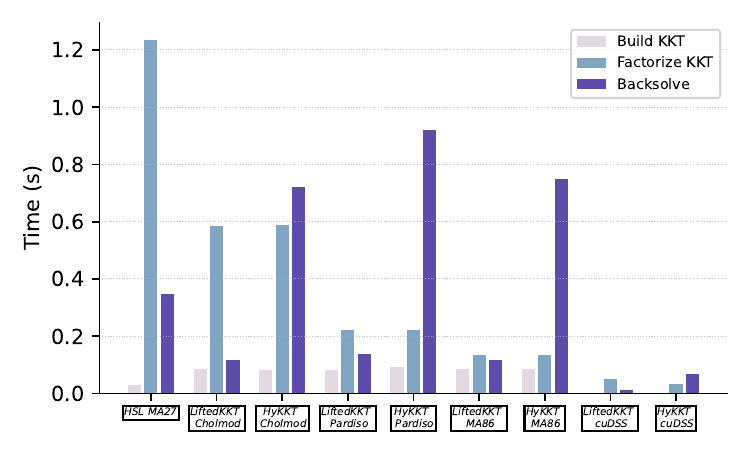}
  \caption{Breakdown of the time spent in one IPM iteration
    for different linear solvers on {\tt 78484epigrids}.
    HSL MA27 factorizing the original augmented KKT system \eqref{eq:kkt:augmented}
    is used as a reference.
    The column for {\tt build} shows the KKT matrix assembly time,
    whereas the columns for {\tt factorize} and {\tt backsolve} display the time
    spent in the numerical factorization and in the triangular solves, respectively.
    Pardiso and HSL MA86 use 8 threads.
  \label{fig:timebreakdown}}
\end{figure}

\subsection{Benchmark on OPF instances}
\label{sec:num:opf}

We benchmark several OPF instances from the PGLIB benchmark~\cite{babaeinejadsarookolaee2019power},
comparing the GPU-based LiftedKKT and HyKKT with the CPU-based augmented KKT system using HSL MA27.
For IPM solvers, we use the MadNLP solver~\cite{shin2021graph} with ExaModels.
We set the IPM tolerance to $10^{-6}$.
Regarding HyKKT, we set $\gamma = 10^7$ based on the analysis in \S\ref{sec:num:tuninghykkt}.
Dolan \& Moré performance profiles~\cite{dolan2002benchmarking} are shown in Figure~\ref{fig:opf:pprof}.
\add{The x-axis displays the performance ratio, measuring how far a given solver is from
  the best solver on each problem.
  The y-axis shows the fraction of test problems solved within a given factor of the best solver.
  The higher the curve, the better the solver.
}
The performance profile on the left compares LiftedKKT and HyKKT with
MadNLP+Ma27. As a baseline, we include the performance obtained by Ipopt and
Knitro, which also run with HSL MA27. CPU performance for HSL MA27
is consistent with~\cite{babaeinejadsarookolaee2019power}.
The right profile focuses on LiftedKKT and HyKKT using
various parallel linear solvers.
For the two condensed strategies, {\tt cuDSS} provides the fastest solution time compared to
Pardiso and HSL MA86 running in parallel on the CPU using 8 threads.

The detailed results are displayed in Table~\ref{tab:opf:benchmark}:
The table displays the time spent on initialization, the time spent in the
linear solver, and the total solving time. On GPUs, LiftedKKT+{\tt cuDSS} is
faster than HyKKT+{\tt cuDSS} on small and medium instances as it avoids the \CG algorithm
at each IPM iteration.
Both outperform HSL MA27 significantly. HyKKT+{\tt cuDSS} is slower
on {\tt 8387\_pegase}: the
parameter $\gamma$ is not set high enough to reduce the total number of \CG
iterations, leading to a 4x slowdown compared to LiftedKKT+{\tt cuDSS}.
Nevertheless, HyKKT+{\tt cuDSS} performance improves on larger
instances, reaching an 8x speedup over HSL MA27.

The benchmark presented in Table~\ref{tab:opf:benchmark} was generated using NVIDIA A100 GPUs (80GB, current selling price: \$10k).
We have also compared the performance
with cheaper GPUs: a NVIDIA A1000 (a laptop-based GPU, 4GB memory) and a NVIDIA A30
(24GB memory, price: \$5k).
As a comparison, the selling price of the AMD EPYC
7443 processor used for the benchmark on the CPU is \$1.2k. The results are
displayed in Figure~\ref{fig:gpubench}. We observe that the performance of the A30
and the A100 is similar, and the cheaper A1000 GPU is already
faster than HSL MA27 running on the CPU.

OPF problems possess a specific sparsity
structure that results in highly sparse KKT
systems~\eqref{eq:kkt:augmented}. The super-sparse structure explains why HSL
MA27 yields better results than HSL MA57 (which can utilize parallelism at
the BLAS level) as previously reported in \cite{tasseff2019exploring}.
Thus, OPF results may not represent general
nonlinear programs. In the next section, we analyze HyKKT and LiftedKKT
performance on the CUTEst benchmark.

\begin{table}[!ht]
  \centering
  \resizebox{\textwidth}{!}{
    \add{
\begin{tabular}{|l|rr|rrrrr>{\bfseries}r|rrrrr>{\bfseries}r|rrrrr>{\bfseries}r|}
  \hline
  & & & \multicolumn{6}{c|}{\bf Ipopt+HSL MA27} &
    \multicolumn{6}{c|}{\bf MadNLP+LiftedKKT+Pardiso} &
    \multicolumn{6}{c|}{\bf MadNLP+HyKKT+Pardiso} \\
                                                    && & \multicolumn{6}{c|}{\bf (CPU)} &
    \multicolumn{6}{c|}{\bf (CPU)} &
    \multicolumn{6}{c|}{\bf (GPU)} \\
  \hline
  & n & m & it & init & AD & lin & time/it & total & it & init & AD & lin & time/it & total & it & init & AD & lin & time/it & total \\
  \hline
1354\_pegase & 11192 & 16646 & 47 & - & - & - & 0.01 & 0.47 & 51 & 0.06 & 0.04 & 0.69 & 0.02 & 1.00 & 48 & 0.06 & 0.04 & 0.45 & 0.01 & 0.68 \\
2000\_goc & 19008 & 29432 & 42 & - & - & - & 0.02 & 0.84 & 40 & 0.11 & 0.06 & 1.05 & 0.04 & 1.53 & 42 & 0.12 & 0.07 & 0.61 & 0.02 & 1.03 \\
2312\_goc & 17128 & 25716 & 44 & - & - & - & 0.02 & 0.75 & 43 & 0.09 & 0.06 & 0.36 & 0.02 & 0.77 & 44 & 0.09 & 0.06 & 0.70 & 0.02 & 1.06 \\
2742\_goc & 24540 & 38196 & 104 & - & - & - & 0.04 & 3.75 & 238 & 0.15 & 0.57 & 5.01 & 0.05 & 10.72 & 91 & 0.15 & 0.27 & 3.27 & 0.08 & 6.86 \\
2869\_pegase & 25086 & 37813 & 57 & - & - & - & 0.03 & 1.50 & 57 & 0.14 & 0.12 & 0.71 & 0.03 & 1.49 & 55 & 0.15 & 0.11 & 1.23 & 0.03 & 1.85 \\
  \hline
3022\_goc & 23238 & 34990 & 55 & - & - & - & 0.02 & 1.25 & 48 & 0.13 & 0.09 & 0.71 & 0.03 & 1.34 & 55 & 0.13 & 0.10 & 1.03 & 0.03 & 1.61 \\
3970\_goc & 35270 & 54428 & 64 & - & - & - & 0.05 & 3.09 & 47 & 0.22 & 0.14 & 1.30 & 0.05 & 2.54 & 48 & 0.26 & 0.14 & 1.35 & 0.05 & 2.45 \\
4020\_goc & 36696 & 56957 & 60 & - & - & - & 0.07 & 4.37 & 59 & 0.23 & 0.18 & 1.78 & 0.06 & 3.32 & 59 & 0.24 & 0.19 & 2.21 & 0.06 & 3.53 \\
4601\_goc & 38814 & 59596 & 72 & - & - & - & 0.05 & 3.81 & 66 & 0.23 & 0.21 & 1.60 & 0.05 & 3.31 & 71 & 0.25 & 0.23 & 2.24 & 0.05 & 3.75 \\
4619\_goc & 42532 & 66289 & 49 & - & - & - & 0.08 & 3.72 & 53 & 0.28 & 0.19 & 2.11 & 0.07 & 3.71 & 49 & 0.29 & 0.18 & 2.01 & 0.07 & 3.36 \\
  \hline
4837\_goc & 41398 & 64030 & 58 & - & - & - & 0.05 & 3.11 & 55 & 0.25 & 0.19 & 1.28 & 0.05 & 2.88 & 59 & 0.26 & 0.21 & 2.04 & 0.06 & 3.48 \\
4917\_goc & 37872 & 56917 & 62 & - & - & - & 0.04 & 2.49 & 70 & 0.22 & 0.23 & 2.52 & 0.06 & 4.31 & 63 & 0.23 & 0.20 & 1.85 & 0.05 & 3.21 \\
5658\_epigrids & 48552 & 74821 & 50 & - & - & - & 0.07 & 3.44 & 59 & 0.30 & 0.25 & 2.00 & 0.07 & 4.01 & 51 & 0.35 & 0.22 & 2.17 & 0.07 & 3.74 \\
7336\_epigrids & 62116 & 95306 & 48 & - & - & - & 0.09 & 4.24 & 53 & 0.44 & 0.28 & 2.02 & 0.08 & 4.31 & 50 & 0.69 & 0.26 & 2.46 & 0.09 & 4.34 \\
8387\_pegase & 78748 & 118702 & 75 & - & - & - & 0.10 & 7.13 & 85 & 0.57 & 0.59 & 3.49 & 0.09 & 7.69 & 75 & 0.86 & 0.53 & 46.94 & 0.67 & 50.11 \\
  \hline
9241\_pegase & 85568 & 130826 & 70 & - & - & - & 0.11 & 7.68 & - & - & - & - & - & - & 70 & 0.89 & 0.53 & 8.14 & 0.17 & 11.68 \\
9591\_goc & 83572 & 130588 & 68 & - & - & - & 0.18 & 12.36 & 65 & 0.87 & 0.48 & 4.11 & 0.12 & 7.98 & 67 & 0.62 & 0.51 & 5.78 & 0.13 & 8.86 \\
10000\_goc & 76804 & 112352 & 81 & - & - & - & 0.11 & 9.17 & 61 & 0.50 & 0.38 & 5.29 & 0.14 & 8.26 & 82 & 0.85 & 0.51 & 4.88 & 0.10 & 8.21 \\
10192\_epigrids & 89850 & 139456 & 58 & - & - & - & 0.17 & 9.70 & 55 & 0.64 & 0.48 & 3.31 & 0.13 & 6.94 & 54 & 0.94 & 0.47 & 5.24 & 0.15 & 8.26 \\
10480\_goc & 96750 & 150874 & 68 & - & - & - & 0.19 & 13.18 & 67 & 0.70 & 0.58 & 5.17 & 0.14 & 9.40 & 71 & 0.72 & 0.62 & 8.14 & 0.17 & 11.85 \\
  \hline
13659\_pegase & 117370 & 170588 & 64 & - & - & - & 0.14 & 9.14 & 77 & 0.86 & 0.79 & 6.77 & 0.16 & 12.47 & 62 & 1.19 & 0.62 & 7.26 & 0.18 & 11.21 \\
19402\_goc & 179562 & 281733 & 69 & - & - & - & 0.50 & 34.56 & 69 & 1.74 & 1.26 & 11.14 & 0.29 & 20.03 & 69 & 2.03 & 1.25 & 17.88 & 0.37 & 25.34 \\
20758\_epigrids & 179236 & 274918 & 48 & - & - & - & 0.33 & 15.83 & 57 & 1.44 & 1.02 & 6.59 & 0.25 & 14.00 & 51 & 1.95 & 0.93 & 12.80 & 0.37 & 18.93 \\
30000\_goc & 208624 & 307752 & 150 & - & - & - & 0.34 & 50.43 & 156 & 1.60 & 3.51 & 30.27 & 0.32 & 50.05 & 223 & 1.63 & 4.77 & 48.31 & 0.32 & 71.06 \\
78484\_epigrids & 674562 & 1039062 & 102 & - & - & - & 1.95 & 198.54 & 103 & 7.30 & 8.18 & 43.05 & 0.95 & 97.44 & 102 & 7.37 & 8.16 & 96.61 & 1.36 & 139.00 \\
  \hline
  \hline
  & & & \multicolumn{6}{c|}{\bf MadNLP+HSL MA27} &
    \multicolumn{6}{c|}{\bf MadNLP+LiftedKKT+cuDSS} &
    \multicolumn{6}{c|}{\bf MadNLP+HyKKT+cuDSS} \\
                                                    && & \multicolumn{6}{c|}{\bf (CPU)} &
    \multicolumn{6}{c|}{\bf (GPU)} &
    \multicolumn{6}{c|}{\bf (GPU)} \\
  \hline
  & n & m & it & init & AD & lin & time/it & total & it & init & AD & lin & time/it & total & it & init & AD & lin & time/it & total \\
  \hline
1354\_pegase & 11192 & 16646 & 48 & 0.01 & 0.04 & 0.33 & 0.01 & 0.46 & 71 & 0.08 & 0.06 & 0.66 & 0.01 & 1.05 & 48 & 0.08 & 0.06 & 0.24 & 0.01 & 0.56 \\
2000\_goc & 19008 & 29432 & 42 & 0.03 & 0.06 & 0.63 & 0.02 & 0.86 & 48 & 0.12 & 0.04 & 0.43 & 0.02 & 0.78 & 42 & 0.12 & 0.06 & 0.23 & 0.01 & 0.57 \\
2312\_goc & 17128 & 25716 & 44 & 0.02 & 0.05 & 0.57 & 0.02 & 0.77 & 63 & 0.12 & 0.05 & 0.45 & 0.01 & 0.90 & 44 & 0.11 & 0.06 & 0.32 & 0.02 & 0.67 \\
2742\_goc & 24540 & 38196 & 122 & 0.04 & 0.28 & 3.55 & 0.06 & 6.90 & 168 & 0.15 & 0.20 & 2.38 & 0.02 & 3.91 & 91 & 0.15 & 0.26 & 1.24 & 0.03 & 2.44 \\
2869\_pegase & 25086 & 37813 & 55 & 0.05 & 0.11 & 1.00 & 0.03 & 1.42 & 57 & 0.15 & 0.05 & 0.33 & 0.01 & 0.80 & 55 & 0.14 & 0.08 & 0.40 & 0.02 & 0.88 \\
  \hline
3022\_goc & 23238 & 34990 & 55 & 0.03 & 0.10 & 0.92 & 0.02 & 1.26 & 53 & 0.13 & 0.06 & 0.30 & 0.01 & 0.74 & 55 & 0.13 & 0.07 & 0.37 & 0.01 & 0.81 \\
3970\_goc & 35270 & 54428 & 48 & 0.06 & 0.14 & 1.89 & 0.06 & 2.64 & 48 & 0.22 & 0.05 & 0.49 & 0.02 & 1.00 & 48 & 0.20 & 0.09 & 0.40 & 0.02 & 0.92 \\
4020\_goc & 36696 & 56957 & 59 & 0.06 & 0.18 & 3.56 & 0.08 & 4.46 & 64 & 0.22 & 0.05 & 0.77 & 0.02 & 1.60 & 59 & 0.21 & 0.12 & 0.67 & 0.02 & 1.33 \\
4601\_goc & 38814 & 59596 & 71 & 0.06 & 0.22 & 3.08 & 0.06 & 4.11 & 67 & 0.21 & 0.06 & 0.60 & 0.02 & 1.22 & 71 & 0.21 & 0.14 & 0.67 & 0.02 & 1.40 \\
4619\_goc & 42532 & 66289 & 49 & 0.07 & 0.18 & 3.07 & 0.08 & 3.95 & 52 & 0.26 & 0.05 & 0.70 & 0.03 & 1.34 & 49 & 0.25 & 0.10 & 0.54 & 0.02 & 1.16 \\
  \hline
4837\_goc & 41398 & 64030 & 59 & 0.06 & 0.20 & 2.44 & 0.06 & 3.39 & 56 & 0.22 & 0.05 & 0.44 & 0.02 & 1.10 & 59 & 0.22 & 0.11 & 0.57 & 0.02 & 1.21 \\
4917\_goc & 37872 & 56917 & 63 & 0.06 & 0.19 & 1.85 & 0.04 & 2.75 & 69 & 0.20 & 0.06 & 0.79 & 0.02 & 1.51 & 63 & 0.20 & 0.10 & 0.54 & 0.02 & 1.16 \\
5658\_epigrids & 48552 & 74821 & 51 & 0.08 & 0.21 & 2.69 & 0.07 & 3.71 & 48 & 0.27 & 0.05 & 0.55 & 0.03 & 1.23 & 51 & 0.26 & 0.11 & 0.56 & 0.02 & 1.22 \\
7336\_epigrids & 62116 & 95306 & 50 & 0.11 & 0.26 & 3.44 & 0.09 & 4.64 & 54 & 0.36 & 0.06 & 1.12 & 0.04 & 2.08 & 50 & 0.37 & 0.11 & 0.58 & 0.03 & 1.36 \\
8387\_pegase & 78748 & 118702 & 74 & 0.15 & 0.48 & 5.19 & 0.10 & 7.18 & 76 & 0.46 & 0.08 & 0.70 & 0.03 & 2.29 & 75 & 0.43 & 0.22 & 8.41 & 0.13 & 9.53 \\
  \hline
9241\_pegase & 85568 & 130826 & 73 & 0.38 & 0.53 & 5.71 & 0.11 & 8.07 & 107 & 0.50 & 0.12 & 1.18 & 0.03 & 3.42 & 70 & 0.49 & 0.14 & 1.68 & 0.04 & 2.76 \\
9591\_goc & 83572 & 130588 & 67 & 0.41 & 0.49 & 10.11 & 0.18 & 12.37 & 75 & 0.49 & 0.09 & 1.99 & 0.05 & 3.74 & 67 & 0.48 & 0.15 & 1.23 & 0.03 & 2.27 \\
10000\_goc & 76804 & 112352 & 82 & 0.16 & 0.50 & 5.65 & 0.09 & 7.69 & 64 & 0.39 & 0.07 & 1.12 & 0.03 & 2.21 & 82 & 0.35 & 0.19 & 1.06 & 0.03 & 2.09 \\
10192\_epigrids & 89850 & 139456 & 54 & 0.42 & 0.46 & 7.18 & 0.17 & 9.42 & 68 & 0.57 & 0.08 & 1.54 & 0.05 & 3.14 & 54 & 0.51 & 0.14 & 1.08 & 0.04 & 2.09 \\
10480\_goc & 96750 & 150874 & 71 & 0.40 & 0.60 & 11.01 & 0.19 & 13.68 & 66 & 0.60 & 0.08 & 1.57 & 0.05 & 3.51 & 71 & 0.56 & 0.14 & 1.61 & 0.04 & 2.74 \\
  \hline
13659\_pegase & 117370 & 170588 & 63 & 0.23 & 0.63 & 6.65 & 0.15 & 9.15 & 68 & 0.69 & 0.09 & 1.09 & 0.04 & 3.03 & 62 & 0.62 & 0.12 & 1.42 & 0.04 & 2.57 \\
19402\_goc & 179562 & 281733 & 69 & 0.65 & 1.23 & 28.47 & 0.48 & 33.19 & 288 & 1.17 & 0.63 & 21.51 & 0.15 & 43.55 & 69 & 1.08 & 0.17 & 2.52 & 0.06 & 4.23 \\
20758\_epigrids & 179236 & 274918 & 51 & 0.37 & 0.91 & 12.85 & 0.32 & 16.32 & 52 & 1.07 & 0.09 & 1.48 & 0.08 & 4.18 & 51 & 1.00 & 0.14 & 1.76 & 0.06 & 3.26 \\
30000\_goc & 208624 & 307752 & 230 & 0.63 & 4.70 & 84.58 & 0.43 & 99.76 & 139 & 0.98 & 0.26 & 4.50 & 0.06 & 8.46 & 215 & 0.90 & 0.56 & 7.40 & 0.05 & 10.46 \\
78484\_epigrids & 674562 & 1039062 & 102 & 2.54 & 7.83 & 164.55 & 1.87 & 191.10 & 110 & 4.84 & 0.44 & 9.90 & 0.24 & 26.01 & 102 & 4.55 & 0.42 & 9.93 & 0.16 & 16.23 \\
  \hline
\end{tabular}
    }
  }
  \caption{OPF benchmark, solved with a tolerance {\tt
    tol=1e-6}.
The derivatives are evaluated using ExaModels.
The column for {\tt init} displays the time spent initializing the IPM solver, including the initial symbolic factorization in the linear solver.
The columns for {\tt AD}, {\tt lin}, {\tt time/it} and {\tt total} display the time spent in the evaluation of the derivatives using AD, the time spent in the linear solver, the average time spent at each iteration, and the total time spent in the IPM algorithm.
   \label{tab:opf:benchmark}}
\end{table}

\begin{figure}[!ht]
  \centering
  \begin{tabular}{cc}
    \includegraphics[width=.5\textwidth]{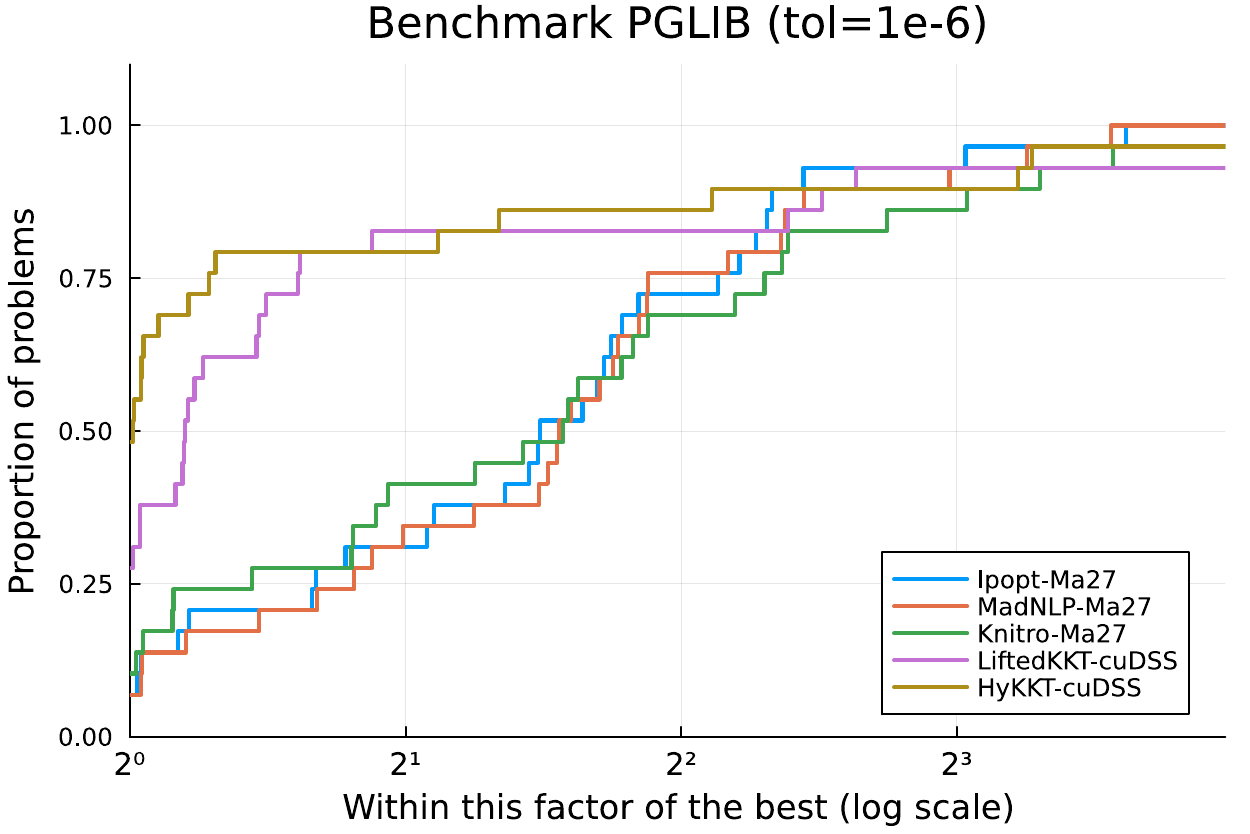} &
    \includegraphics[width=.5\textwidth]{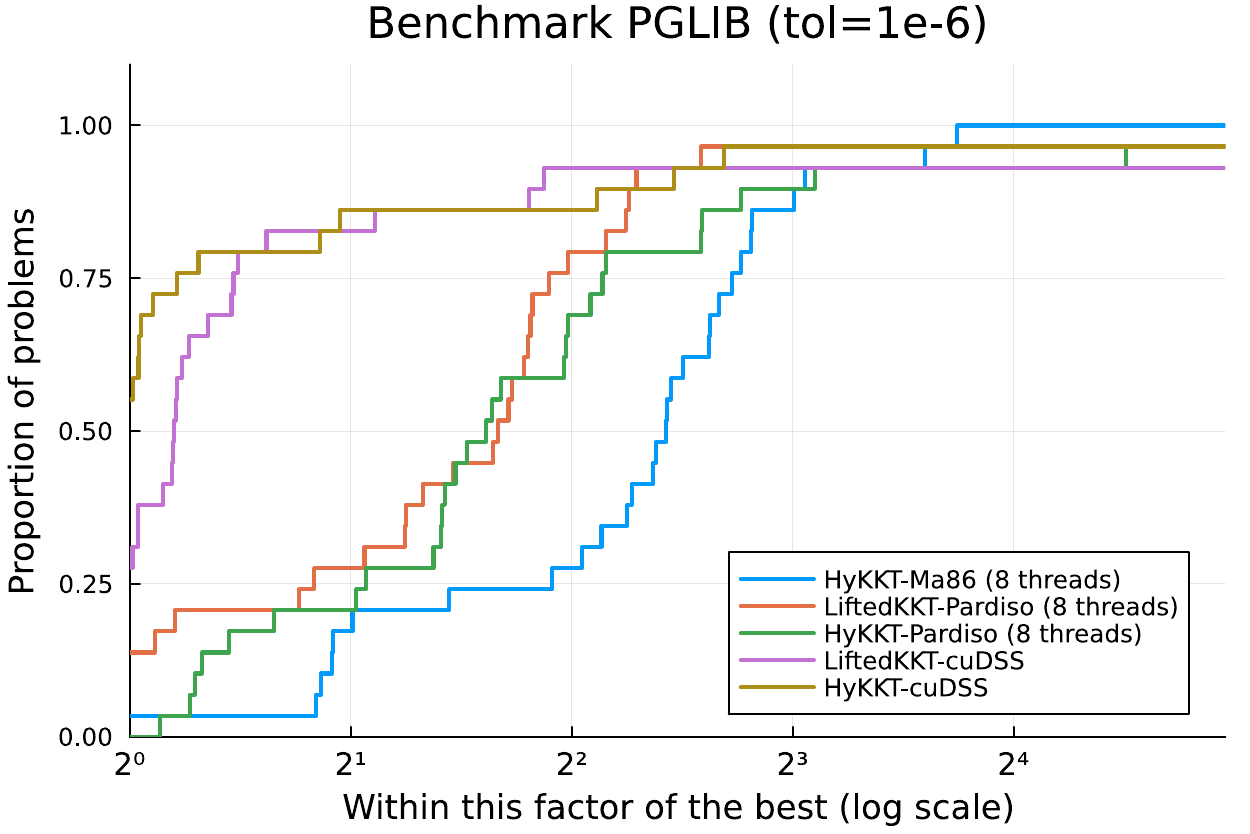}
  \end{tabular}
  \caption{Performance profile \add{for 29 instances of the} PGLIB OPF benchmark, solved
    with a tolerance {\tt tol=1e-6}. Left: comparing LiftedKKT+{\tt cuDSS} and
    HyKKT+{\tt cuDSS} with HSL MA27. Right: comparing LiftedKKT and HyKKT
    with various parallel sparse linear solvers.
  \label{fig:opf:pprof}}
\end{figure}

\begin{figure}[!ht]
  \centering
  \includegraphics[width=.6\textwidth]{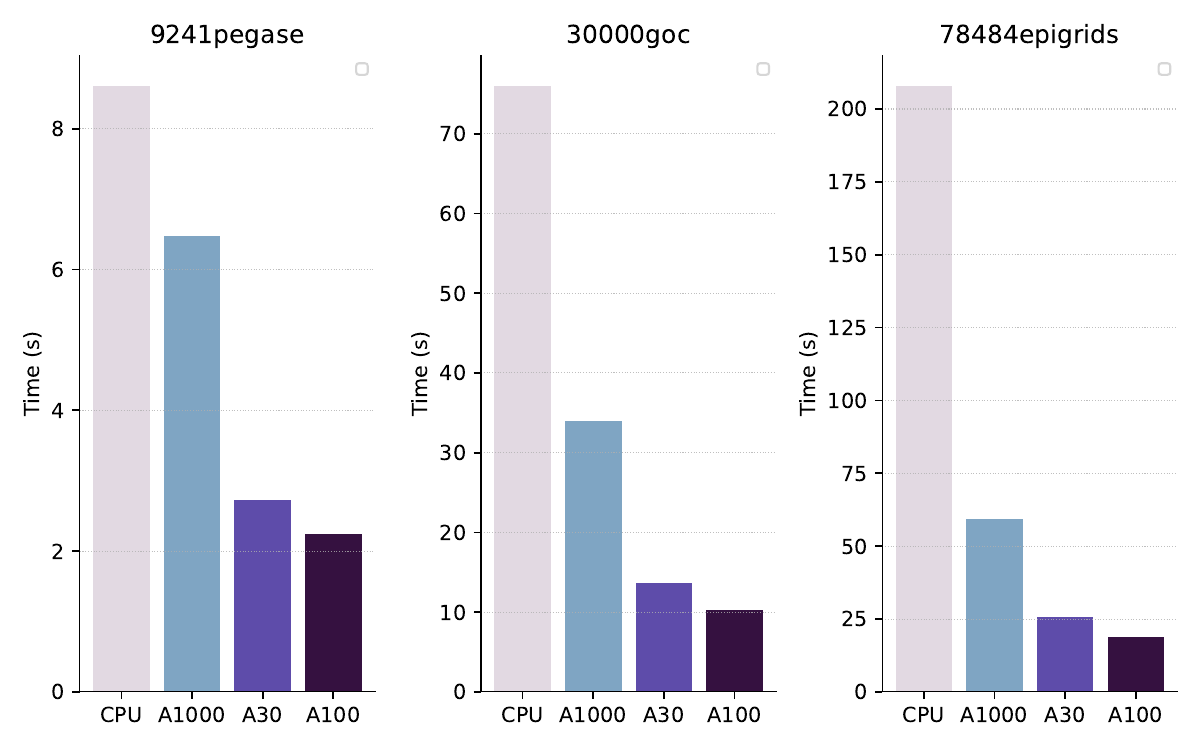}
  \caption{Performance on three OPF instances across different GPUs using HyKKT.
  \label{fig:gpubench}}
\end{figure}

\subsection{Benchmark on CUTEst instances}
\label{sec:num:cutest}
We extend our analysis to large-scale CUTEst instances with more than 1,000 variables and at least one constraint. We use an IPM tolerance of $10^{-6}$ and set $\gamma = 10^7$ for HyKKT. SIF instances are decoded on the CPU using {\tt CUTEst.jl}, \add{and models are evaluated on the CPU entirely. To use CUTEst on GPUs with HyKKT+{\tt cuDSS} and LiftedKKT+{\tt cuDSS}, we have implemented a wrapper that copies sensitivities (gradient, constraints, Jacobian, and Hessian) to the GPU. Consequently, each IPM iteration requires copying four arrays to the GPU, and AD evaluations are slower than on-GPU evaluations.}

Performance profiles comparing LiftedKKT+{\tt cuDSS} and HyKKT+{\tt cuDSS} are
in Figure~\ref{fig:cutest:pprof}. As a baseline, we include the performance
obtained with Ipopt, MadNLP, and Knitro, all running with HSL MA57 (in
contrast to OPF problems, HSL MA57 has better performance than HSL MA27 on CUTEst).
In the performance profile (a), we observe that HyKKT and LiftedKKT are less
competitive than a standard method running on the CPU, despite demonstrating
robustness comparable to Knitro for LiftedKKT. However, if we select only the
largest instances where Ipopt solve time is greater than 1 second
(performance profile (b)), HyKKT and LiftedKKT become more competitive, matching
Knitro performance on 50\% of the instances.

Overall, HyKKT is less robust than LiftedKKT on CUTEst: in some cases, the
regularization parameter $\gamma = 10^7$ is insufficient to ensure that
$K_\gamma$ is positive definite. This is the case for the {\tt
ROCKET} problem: this optimal control instance is known to be degenerate, in the sense that
the reduced Hessian $Z^\top K Z$ is close to being singular (this is induced
by the infamous singular arc in Goddard's problem). In that particular instance, we have observed that
$\gamma$ must be increased to $10^{12}$ to obtain a positive definite matrix $K_\gamma$.

Figure~\ref{fig:cutest:pprof} shows that the two condensed space methods
are less robust than the reference methods running on the CPU.
Table~\ref{tab:cutest:benchmark} presents a
selection of representative instances where condensed-space methods are
either significantly slower (unfavorable cases) or faster (favorable cases) than the references.
Unfavorable cases typically involve: (i) significant {\tt cuDSS} symbolic
factorization time (e.g., {\tt BLOCKQP2}, {\tt ROSEPETAL2}),
(explaining the bump observed at the upper right of the performance profiles in Figure~\ref{fig:cutest:pprof});
(ii) dense or nearly dense columns in the Jacobian ({\tt ROCKET}, {\tt BRAINPC0}, {\tt BLOCKQP2}, {\tt ROSEPETAL2}, {\tt A5NSDSIL}, {\tt A0NNDNIL}, {\tt A0NSDSIL}, {\tt A0NNSNSL}, {\tt MPC9}, {\tt CMPC10}, {\tt CMPC12},
{\tt DALE}, {\tt CBS}), leading to additional fill-in
when factorizing the condensed matrices $K_\gamma$ and $K_\tau$;
(iii) poor IPM convergence for the equality relaxation strategy ({\tt LUKVLE8}, {\tt READING7}).
However, we observe on the favorable cases that if the problem does not suffer from one of the previous pitfalls,
HybridKKT and LiftedKKT become competitive w.r.t. the references, with speed-up
similar to what we observed before on the OPF instances (Table~\ref{tab:opf:benchmark}).
We observe that the condensed methods are efficient to solve problems
with a large number of non-zero coefficients in the Jacobian and in the Hessian
({\tt BDRY2}, {\tt CLEUVEN*})
or problems with a dynamical structure ({\tt MARINE}, {\tt ROBOTARM}).
Interestingly, the condensed-space methods can
accommodate dense condensed matrices $K_\tau$ and $K_\gamma$
if they remain sufficiently small ({\tt EIGENAU}, {\tt EIGENCCO}, {\tt EIGENB2}, {\tt MSQRTB}, where the size of the dense condensed matrix is at most $\approx 2,500 \times 2,500$, compared to more than $10,000 \times 10,000$ in the unfavorable cases).

\begin{figure}[!ht]
  \centering
  \begin{tabular}{cc}
    \includegraphics[width=.5\textwidth]{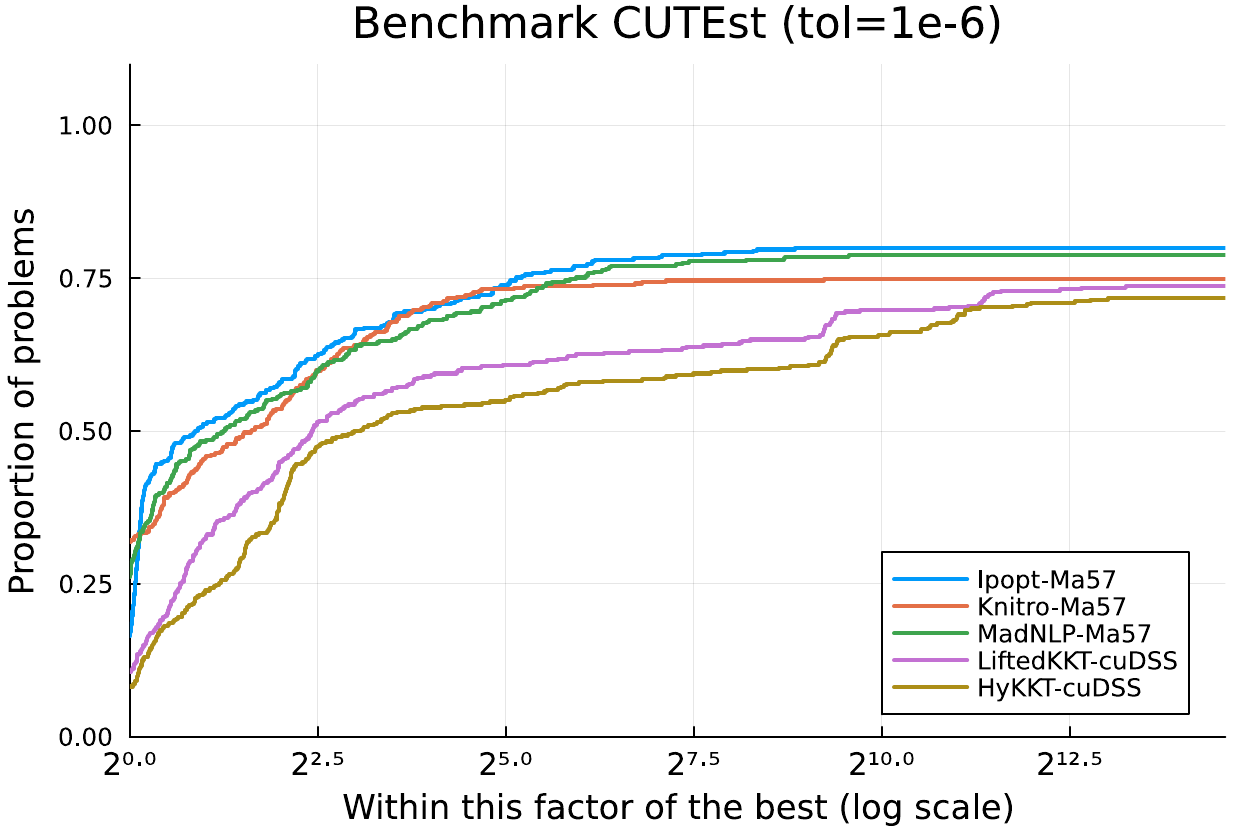} &
    \includegraphics[width=.5\textwidth]{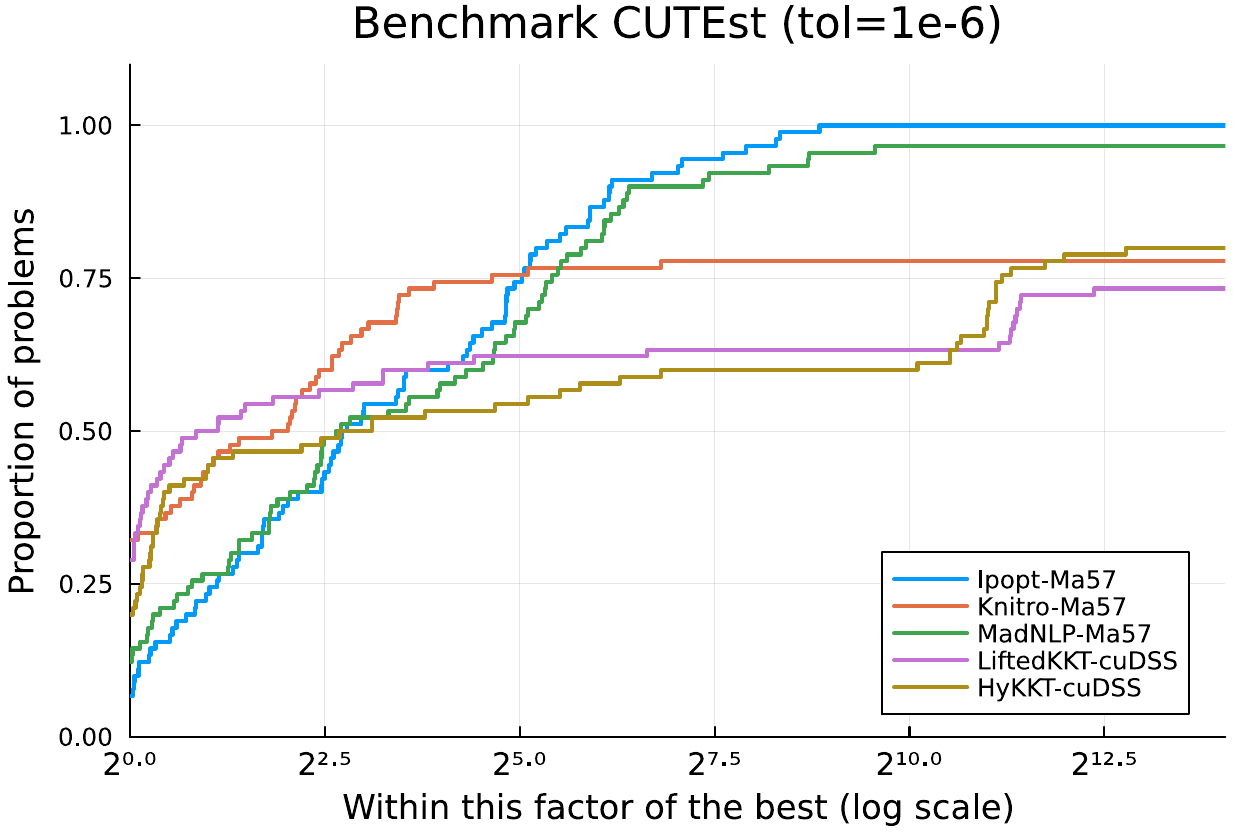}\\
    (a) & (b)
  \end{tabular}
  \caption{
    Performance profiles for 413 CUTEst instances ($>$1,000 variables, $>$1
    constraint). Profile (a) shows the full benchmark; (b) shows only
    instances where Ipopt takes $>$1 s.
    \label{fig:cutest:pprof}}
\end{figure}

\begin{table}[!ht]
  \centering
  \resizebox{\textwidth}{!}{
\begin{tabular}{|l|rr|rrrrr>{\bfseries}r|rrrrr>{\bfseries}r|rrrrr>{\bfseries}r|}
\toprule
  & \multicolumn{2}{c|}{}
  & \multicolumn{6}{c|}{\bf MadNLP+HSL MA57} &
  \multicolumn{6}{c|}{\bf MadNLP+LiftedKKT+cuDSS} &
\multicolumn{6}{c|}{\bf MadNLP+HyKKT+cuDSS} \\
\midrule
  & n & m & it & init & AD & lin & time/it & total & it & init & AD & lin &time/it &  total & it & init & AD & lin &time/it &  total \\
\midrule
 & \multicolumn{20}{c|}{\bf GPU-accelerated solvers are significantly slower} \\
\midrule
A5NSDSIL & 60012 & 20004 & 265.00 & 0.03 & 1.22 & 6.97 & 0.04 & 11.33 & 119.00 & 26.16 & 1.35 & 129.95 & 1.38 & 164.00 & - & - & - & - & - & - \\
A0NNDNIL & 60012 & 20004 & 97.00 & 0.02 & 0.43 & 2.22 & 0.03 & 3.15 & 470.00 & 17.39 & 4.10 & 365.43 & 0.84 & 395.94 & 97.00 & 17.18 & 0.93 & 73.82 & 0.96 & 93.57 \\
A0NSDSIL & 60012 & 20004 & 100.00 & 0.02 & 0.43 & 2.23 & 0.03 & 3.15 & 212.00 & 18.41 & 1.88 & 168.46 & 0.91 & 192.80 & 100.00 & 18.50 & 1.20 & 79.08 & 1.01 & 100.67 \\
A0NNSNSL & 60012 & 20004 & 31.00 & 0.04 & 0.14 & 1.36 & 0.05 & 1.62 & 36.00 & 19.00 & 0.34 & 38.76 & 1.63 & 58.78 & 38.00 & 19.24 & 0.48 & 80.65 & 2.66 & 101.15 \\
CMPC10 & 1530 & 2351 & 70.00 & 0.00 & 0.02 & 0.24 & 0.00 & 0.27 & 82.00 & 0.10 & 0.29 & 1.93 & 0.06 & 4.90 & 70.00 & 0.04 & 0.06 & 0.78 & 0.02 & 1.07 \\
MPC9 & 1530 & 2351 & 70.00 & 0.00 & 0.02 & 0.25 & 0.00 & 0.29 & 79.00 & 0.21 & 0.63 & 3.25 & 0.12 & 9.78 & 69.00 & 0.04 & 0.07 & 0.75 & 0.02 & 1.05 \\
LUKVLE8 & 10000 & 9998 & 12.00 & 0.01 & 0.05 & 0.04 & 0.01 & 0.10 & 144.00 & 0.12 & 1.36 & 2.48 & 0.18 & 25.83 & 12.00 & 0.04 & 0.05 & 13.45 & 1.13 & 13.57 \\
READING7 & 1002 & 500 & 125.00 & 0.01 & 0.14 & 1.28 & 0.01 & 1.49 & 330.00 & 0.65 & 1.45 & 4.75 & 0.06 & 18.18 & 125.00 & 0.55 & 0.35 & 1.25 & 0.02 & 3.09 \\
BRAINPC0 & 6907 & 6900 & 120.00 & 0.01 & 0.27 & 0.78 & 0.01 & 1.11 & 43.00 & 0.22 & 0.47 & 2.83 & 0.18 & 7.65 & - & - & - & - & - & - \\
ORTHRDS2C & 5003 & 2500 & - & - & - & - & - & - & 30.00 & 0.22 & 0.37 & 1.09 & 0.14 & 4.16 & - & - & - & - & - & - \\
ROCKET & 2407 & 2002 & 58.00 & 0.00 & 0.05 & 0.10 & 0.00 & 0.16 & 38.00 & 0.21 & 0.35 & 1.71 & 0.13 & 5.11 & - & - & - & - & - & - \\
HYDROELL & 1009 & 1008 & 235.00 & 0.00 & 0.03 & 0.08 & 0.00 & 0.15 & 235.00 & 0.05 & 0.41 & 0.70 & 0.02 & 4.65 & 235.00 & 0.02 & 0.16 & 0.19 & 0.00 & 0.95 \\
DALE & 16514 & 405 & 14.00 & 0.01 & 0.01 & 0.07 & 0.01 & 0.11 & 16.00 & 2.12 & 0.05 & 30.59 & 2.06 & 33.03 & 14.00 & 1.66 & 0.03 & 16.60 & 1.31 & 18.36 \\
CBS & 11163 & 244 & 20.00 & 0.00 & 0.01 & 0.07 & 0.00 & 0.10 & 22.00 & 0.75 & 0.03 & 7.60 & 0.39 & 8.49 & 35.00 & 0.74 & 0.07 & 8.92 & 0.28 & 9.91 \\
BLOCKQP2 & 10010 & 5001 & 19.00 & 0.01 & 0.02 & 0.09 & 0.01 & 0.13 & 37.00 & 36.27 & 0.08 & 7.37 & 1.21 & 44.62 & 19.00 & 36.30 & 0.05 & 2.75 & 2.08 & 39.50 \\
ROSEPETAL2 & 10001 & 20001 & 29.00 & 0.05 & 0.04 & 0.15 & 0.01 & 0.29 & 47.00 & 36.11 & 0.17 & 7.41 & 0.95 & 44.86 & 29.00 & 36.15 & 0.17 & 3.24 & 1.38 & 40.16 \\
\midrule
 & \multicolumn{20}{c|}{\bf GPU-accelerated solvers are significantly faster} \\
\midrule
STCQP1 & 8193 & 4095 & 13.00 & 0.01 & 0.06 & 13.87 & 1.07 & 13.96 & 15.00 & 0.06 & 0.08 & 0.06 & 0.02 & 0.26 & 161.00 & 0.07 & 1.10 & 1.18 & 0.02 & 2.94 \\
CVXQP3 & 10000 & 7500 & 17.00 & 0.12 & 0.03 & 36.33 & 2.15 & 36.50 & 31.00 & 0.13 & 0.07 & 0.41 & 0.02 & 0.69 & 17.00 & 0.13 & 0.07 & 118.33 & 6.98 & 118.60 \\
CATENA & 3003 & 1000 & 32.00 & 0.00 & 0.02 & 12.72 & 0.40 & 12.74 & 36.00 & 0.06 & 0.10 & 0.33 & 0.03 & 1.13 & - & - & - & - & - & - \\
CLEUVEN3 & 1200 & 2973 & 134.00 & 0.08 & 0.80 & 6.81 & 0.06 & 7.87 & 104.00 & 0.53 & 0.80 & 0.29 & 0.02 & 2.16 & 99.00 & 0.54 & 0.84 & 1.07 & 0.03 & 3.00 \\
NCVXQP7 & 10000 & 7500 & 121.00 & 0.12 & 0.21 & 133.73 & 1.11 & 134.46 & 302.00 & 0.13 & 1.12 & 7.91 & 0.03 & 10.55 & 129.00 & 0.13 & 0.48 & 178.30 & 1.39 & 179.46 \\
BDRY2 & 251001 & 250498 & 45.00 & 2.53 & 1.06 & 158.10 & 3.64 & 163.68 & 157.00 & 1.46 & 4.99 & 28.96 & 0.23 & 36.30 & 50.00 & 1.75 & 1.57 & 10.99 & 0.30 & 14.77 \\
SCW2 & 2003 & 1001 & 251.00 & 0.41 & 25.96 & 66.30 & 0.38 & 95.12 & 97.00 & 0.75 & 10.64 & 1.55 & 0.14 & 13.48 & 259.00 & 0.74 & 28.03 & 4.71 & 0.14 & 35.14 \\
CLEUVEN6 & 1200 & 3091 & 206.00 & 0.08 & 1.26 & 9.45 & 0.05 & 11.05 & 88.00 & 0.55 & 0.74 & 0.27 & 0.02 & 2.04 & - & - & - & - & - & - \\
AUG3D & 27543 & 8000 & 2.00 & 0.16 & 0.01 & 1.45 & 0.81 & 1.61 & 3.00 & 0.22 & 0.01 & 0.04 & 0.10 & 0.29 & 1.00 & 0.22 & 0.01 & 0.03 & 0.26 & 0.26 \\
ROBOTARM & 4412 & 3202 & 451.00 & 0.00 & 0.28 & 1.38 & 0.00 & 1.81 & 36.00 & 0.06 & 0.04 & 0.18 & 0.01 & 0.46 & - & - & - & - & - & - \\
EIGENB2 & 2550 & 1275 & 132.00 & 0.08 & 0.82 & 211.36 & 1.61 & 212.42 & 110.00 & 1.23 & 1.01 & 1.77 & 0.04 & 4.68 & 128.00 & 1.23 & 1.21 & 2.14 & 0.04 & 5.25 \\
CBRATU3D & 3456 & 2000 & 3.00 & 0.03 & 0.00 & 0.29 & 0.11 & 0.32 & 5.00 & 0.06 & 0.01 & 0.03 & 0.03 & 0.13 & 3.00 & 0.07 & 0.01 & 0.03 & 0.04 & 0.12 \\
EIGENC2 & 2652 & 1326 & 16.00 & 0.08 & 0.10 & 26.37 & 1.66 & 26.57 & 17.00 & 1.31 & 0.16 & 0.58 & 0.13 & 2.15 & 15.00 & 1.32 & 0.16 & 0.56 & 0.14 & 2.12 \\
STNQP1 & 8193 & 4095 & 17.00 & 0.01 & 0.07 & 19.59 & 1.16 & 19.69 & 21.00 & 0.06 & 0.12 & 0.13 & 0.02 & 0.38 & 60.00 & 0.06 & 0.40 & 0.48 & 0.02 & 1.24 \\
EIGENCCO & 2652 & 1326 & 42.00 & 0.58 & 8.85 & 51.38 & 1.46 & 61.37 & 35.00 & 1.46 & 7.56 & 0.59 & 0.28 & 9.84 & 42.00 & 1.77 & 9.09 & 0.83 & 0.28 & 11.93 \\
EIGENAU & 2550 & 2550 & 2.00 & 0.05 & 0.03 & 10.95 & 5.52 & 11.05 & 3.00 & 1.49 & 0.05 & 0.02 & 0.53 & 1.60 & 1.00 & 1.50 & 0.03 & 0.01 & 1.57 & 1.57 \\
LUKVLE12 & 9997 & 7497 & 86.00 & 0.00 & 0.18 & 208.54 & 2.43 & 208.77 & 30.00 & 0.05 & 0.09 & 0.09 & 0.01 & 0.41 & 24.00 & 0.04 & 0.07 & 0.08 & 0.01 & 0.30 \\
MSQRTB & 1024 & 1024 & 5.00 & 0.01 & 0.01 & 1.62 & 0.33 & 1.65 & 6.00 & 0.23 & 0.02 & 0.03 & 0.05 & 0.30 & 5.00 & 0.23 & 0.02 & 0.02 & 0.06 & 0.29 \\
DRCAVTY1 & 4489 & 3969 & 8.00 & 0.01 & 0.07 & 1.19 & 0.16 & 1.28 & 8.00 & 0.08 & 0.09 & 0.05 & 0.03 & 0.25 & 7.00 & 0.08 & 0.09 & 0.05 & 0.03 & 0.24 \\
CLEUVEN5 & 1200 & 2973 & 134.00 & 0.07 & 0.81 & 6.80 & 0.06 & 7.85 & 104.00 & 0.53 & 0.82 & 0.29 & 0.02 & 2.18 & 99.00 & 0.53 & 0.77 & 1.11 & 0.03 & 3.04 \\
MARINE & 11215 & 11192 & 85.00 & 0.01 & 0.11 & 8.60 & 0.10 & 8.75 & 25.00 & 0.09 & 0.05 & 0.14 & 0.02 & 0.39 & 31.00 & 0.10 & 0.07 & 0.25 & 0.02 & 0.54 \\
\midrule
\bottomrule
\end{tabular}
  }
  \caption{CUTEst benchmark results ({\tt tol=2e-6}). Sensitivities are
    evaluated with CUTEst. The ``init'' column shows initialization time;
    ``AD'', ``lin'', ``time/it'', and ``total'' columns show AD evaluation,
    linear solver, per-iteration, and total IPM times, respectively.
  \label{tab:cutest:benchmark}}
\end{table}



\section{Conclusion}
This paper investigates the theoretical and numerical properties of two condensed-space interior-point methods for solving large-scale nonlinear programming problems on GPUs.
Despite the emergence of ill-conditioned matrices, our theoretical analysis shows that the resulting loss of accuracy is benign in floating-point arithmetic, owing to the structural properties of the interior-point method.
Numerically, both methods prove to be competitive for solving large-scale nonlinear problems.
The fully GPU-resident software stack we developed, which includes MadNLP, cuDSS, and ExaModels, demonstrates its effectiveness on large-scale optimal power flow (OPF) problems from the PGLIB-OPF library,
with up to a 10× speedup on large-scale OPF problems.
While performance is more variable across the CUTEst benchmark, we have identified the structure impairing the performance of LiftedKKT and HyKKT.
These results highlight both the potential (high performance and scalability) and the limitations (robustness and convergence) of GPU-based interior-point methods for nonlinear programming.
Looking ahead, we plan to improve the robustness of the two condensed KKT methods, with a particular emphasis on achieving stable convergence at tighter tolerances (below $10^{-8}$).
Moreover, the sparse Cholesky solver can be further customized to align with the specific structure and needs of interior-point algorithms~\cite{wright1999modified}.

\section*{Funding}
This research used resources of the Argonne Leadership Computing Facility, a U.S. Department of Energy (DOE) Office of Science user facility at Argonne National Laboratory and is based on research supported by the U.S. DOE Office of Science-Advanced Scientific Computing Research Program, under Contract No. DE-AC02-06CH11357.

\section*{Competing interests}
The authors declare they have no competing interests.


\small

\begin{thebibliography}{10}
\providecommand{\url}[1]{{#1}}
\providecommand{\urlprefix}{URL }
\expandafter\ifx\csname urlstyle\endcsname\relax
  \providecommand{\doi}[1]{DOI~\discretionary{}{}{}#1}\else
  \providecommand{\doi}{DOI~\discretionary{}{}{}\begingroup \urlstyle{rm}\Url}\fi

\bibitem{amestoy-david-duff-2004}
Amestoy, P.R., Davis, T.A., Duff, I.S.: {Algorithm 837: AMD, an approximate minimum degree ordering algorithm}.
\newblock ACM Transactions on Mathematical Software (TOMS) \textbf{30}(3), 381--388 (2004)

\bibitem{amos2017optnet}
Amos, B., Kolter, J.Z.: Optnet: Differentiable optimization as a layer in neural networks.
\newblock In: International Conference on Machine Learning, pp. 136--145. PMLR (2017)

\bibitem{babaeinejadsarookolaee2019power}
Babaeinejadsarookolaee, S., Birchfield, A., Christie, R.D., Coffrin, C., DeMarco, C., Diao, R., Ferris, M., Fliscounakis, S., Greene, S., Huang, R., et~al.: The power grid library for benchmarking {AC} optimal power flow algorithms.
\newblock arXiv preprint arXiv:1908.02788  (2019)

\bibitem{benzi2005numerical}
Benzi, M., Golub, G.H., Liesen, J.: Numerical solution of saddle point problems.
\newblock Acta numerica \textbf{14}, 1--137 (2005)

\bibitem{bezanson-edelman-karpinski-shah-2017}
Bezanson, J., Edelman, A., Karpinski, S., Shah, V.B.: Julia: A fresh approach to numerical computing.
\newblock SIAM Review \textbf{59}(1), 65--98 (2017).
\newblock \doi{10.1137/141000671}

\bibitem{bischof1991exploiting}
Bischof, C., Griewank, A., Juedes, D.: Exploiting parallelism in automatic differentiation.
\newblock In: Proceedings of the 5th international conference on Supercomputing, pp. 146--153 (1991)

\bibitem{jax2018github}
Bradbury, J., Frostig, R., Hawkins, P., Johnson, M.J., Leary, C., Maclaurin, D., Necula, G., Paszke, A., Vander{P}las, J., Wanderman-{M}ilne, S., Zhang, Q.: {JAX}: composable transformations of {P}ython+{N}um{P}y programs (2018).
\newblock \urlprefix\url{http://github.com/google/jax}

\bibitem{cao2016augmented}
Cao, Y., Seth, A., Laird, C.D.: An augmented lagrangian interior-point approach for large-scale {NLP} problems on graphics processing units.
\newblock Computers \& Chemical Engineering \textbf{85}, 76--83 (2016)

\bibitem{curtisNoteImplementationInteriorpoint2012}
Curtis, F.E., Huber, J., Schenk, O., W{\"a}chter, A.: A note on the implementation of an interior-point algorithm for nonlinear optimization with inexact step computations.
\newblock Mathematical Programming \textbf{136}(1), 209--227 (2012).
\newblock \doi{10.1007/s10107-012-0557-4}

\bibitem{debreu1952definite}
Debreu, G.: Definite and semidefinite quadratic forms.
\newblock Econometrica: Journal of the Econometric Society pp. 295--300 (1952)

\bibitem{dolan2002benchmarking}
Dolan, E.D., Mor{\'e}, J.J.: Benchmarking optimization software with performance profiles.
\newblock Mathematical programming \textbf{91}, 201--213 (2002)

\bibitem{duff2004ma57}
Duff, I.S.: {MA57}---a code for the solution of sparse symmetric definite and indefinite systems.
\newblock ACM Transactions on Mathematical Software (TOMS) \textbf{30}(2), 118--144 (2004)

\bibitem{duff1983multifrontal}
Duff, I.S., Reid, J.K.: The multifrontal solution of indefinite sparse symmetric linear.
\newblock ACM Transactions on Mathematical Software (TOMS) \textbf{9}(3), 302--325 (1983)

\bibitem{fourer1990ampl}
Fourer, R., Gay, D.M., Kernighan, B.W.: {AMPL: A mathematical programming language}.
\newblock Management Science \textbf{36}(5), 519--554 (1990)

\bibitem{fowkes-lister-montoison-orban-2024}
Fowkes, J., Lister, A., Montoison, A., Orban, D.: {LibHSL: the ultimate collection for large-scale scientific computation}.
\newblock {Les Cahiers du GERAD} G-2024-06, Groupe d’études et de recherche en analyse des décisions (2024)

\bibitem{golub2003solving}
Golub, G.H., Greif, C.: On solving block-structured indefinite linear systems.
\newblock SIAM Journal on Scientific Computing \textbf{24}(6), 2076--2092 (2003)

\bibitem{gondzio-2012}
Gondzio, J.: Interior point methods 25 years later.
\newblock European Journal of Operational Research \textbf{218}(3), 587--601 (2012).
\newblock \doi{10.1016/j.ejor.2011.09.017}

\bibitem{gould2015cutest}
Gould, N.I., Orban, D., Toint, P.L.: Cutest: a constrained and unconstrained testing environment with safe threads for mathematical optimization.
\newblock Computational optimization and applications \textbf{60}, 545--557 (2015)

\bibitem{gould2010spectral}
Gould, N.I., Simoncini, V.: Spectral analysis of saddle point matrices with indefinite leading blocks.
\newblock SIAM Journal on Matrix Analysis and Applications \textbf{31}(3), 1152--1171 (2010)

\bibitem{hestenes-stiefel-1952}
Hestenes, M.R., Stiefel, E.: Methods of conjugate gradients for solving linear systems.
\newblock Journal of Research of the National Bureau of Standards \textbf{49}(6), 409--436 (1952).
\newblock \doi{10.6028/jres.049.044}

\bibitem{hijazi2018gravity}
Hijazi, H., Wang, G., Coffrin, C.: Gravity: A mathematical modeling language for optimization and machine learning  (2018)

\bibitem{kimLeveragingGPUBatching2021}
Kim, Y., Pacaud, F., Kim, K., Anitescu, M.: Leveraging {{GPU}} batching for scalable nonlinear programming through massive {{Lagrangian}} decomposition (2021).
\newblock \doi{10.48550/arXiv.2106.14995}

\bibitem{luCuPDLPFurtherEnhanced2025}
Lu, H., Peng, Z., Yang, J.: {{cuPDLP}}+: {{A Further Enhanced GPU-Based First-Order Solver}} for {{Linear Programming}} (2025).
\newblock \doi{10.48550/arXiv.2507.14051}

\bibitem{lu2023cupdlp}
Lu, H., Yang, J.: {cuPDLP.jl}: A {GPU} implementation of restarted primal-dual hybrid gradient for linear programming in julia.
\newblock arXiv preprint arXiv:2311.12180  (2023)

\bibitem{lu2023cupdlp2}
Lu, H., Yang, J., Hu, H., Huangfu, Q., Liu, J., Liu, T., Ye, Y., Zhang, C., Ge, D.: {cuPDLP-C}: A strengthened implementation of {cuPDLP} for linear programming by {C} language.
\newblock arXiv preprint arXiv:2312.14832  (2023)

\bibitem{montoison2023krylov}
Montoison, A., Orban, D.: Krylov. jl: A julia basket of hand-picked {K}rylov methods.
\newblock Journal of Open Source Software \textbf{8}(89), 5187 (2023)

\bibitem{montoison-orban-hsl-2021}
Montoison, A., Orban, D., {contributors}: {HSL.jl}: A {J}ulia interface to the {HSL} mathematical software library.
\newblock \url{https://github.com/JuliaSmoothOptimizers/HSL.jl} (2021).
\newblock \doi{10.5281/zenodo.2658672}

\bibitem{enzyme2021}
Moses, W.S., Churavy, V., Paehler, L., H\"{u}ckelheim, J., Narayanan, S.H.K., Schanen, M., Doerfert, J.: Reverse-mode automatic differentiation and optimization of {GPU} kernels via {Enzyme}.
\newblock In: Proceedings of the International Conference for High Performance Computing, Networking, Storage and Analysis, SC '21. Association for Computing Machinery, New York, NY, USA (2021).
\newblock \doi{10.1145/3458817.3476165}.
\newblock \urlprefix\url{https://doi.org/10.1145/3458817.3476165}

\bibitem{nocedal_numerical_2006}
Nocedal, J., Wright, S.J.: Numerical optimization, 2nd edn.
\newblock Springer series in operations research. Springer, New York (2006)

\bibitem{nvidiaNVIDIACuDSSPreview}
NVIDIA: {{NVIDIA cuDSS}} ({{Preview}}): {{A}} high-performance {{CUDA Library}} for {{Direct Sparse Solvers}} --- {{NVIDIA cuDSS}} documentation.
\newblock https://docs.nvidia.com/cuda/cudss/index.html

\bibitem{pacaud2022condensed}
Pacaud, F., Shin, S., Schanen, M., Maldonado, D.A., Anitescu, M.: Accelerating condensed interior-point methods on {SIMD/GPU} architectures.
\newblock Journal of Optimization Theory and Applications pp. 1--20 (2023)

\bibitem{pineda2022theseus}
Pineda, L., Fan, T., Monge, M., Venkataraman, S., Sodhi, P., Chen, R.T., Ortiz, J., DeTone, D., Wang, A., Anderson, S., et~al.: Theseus: A library for differentiable nonlinear optimization.
\newblock Advances in Neural Information Processing Systems \textbf{35}, 3801--3818 (2022)

\bibitem{regev2023hykkt}
Regev, S., Chiang, N.Y., Darve, E., Petra, C.G., Saunders, M.A., {\'S}wirydowicz, K., Pele{\v{s}}, S.: {HyKKT}: a hybrid direct-iterative method for solving {KKT} linear systems.
\newblock Optimization Methods and Software \textbf{38}(2), 332--355 (2023)

\bibitem{rodriguezScalablePreconditioningBlockstructured2020}
Rodriguez, J.S., Laird, C.D., Zavala, V.M.: Scalable preconditioning of block-structured linear algebra systems using {{ADMM}}.
\newblock Computers \& Chemical Engineering \textbf{133}, 106478 (2020).
\newblock \doi{10.1016/j.compchemeng.2019.06.003}

\bibitem{schubigerGPUAccelerationADMM2020}
Schubiger, M., Banjac, G., Lygeros, J.: {{GPU}} acceleration of {{ADMM}} for large-scale quadratic programming.
\newblock Journal of Parallel and Distributed Computing \textbf{144}, 55--67 (2020).
\newblock \doi{10.1016/j.jpdc.2020.05.021}

\bibitem{shin2021graph}
Shin, S., Coffrin, C., Sundar, K., Zavala, V.M.: Graph-based modeling and decomposition of energy infrastructures.
\newblock IFAC-PapersOnLine \textbf{54}(3), 693--698 (2021)

\bibitem{shin2023accelerating}
Shin, S., Pacaud, F., Anitescu, M.: Accelerating optimal power flow with {GPUs}: {SIMD} abstraction of nonlinear programs and condensed-space interior-point methods.
\newblock arXiv preprint arXiv:2307.16830  (2023)

\bibitem{stewart1990matrix}
Stewart, G.W., Sun, J.g.: Matrix perturbation theory.
\newblock (No Title)  (1990)

\bibitem{swirydowicz2021linear}
{\'S}wirydowicz, K., Darve, E., Jones, W., Maack, J., Regev, S., Saunders, M.A., Thomas, S.J., Pele{\v{s}}, S.: Linear solvers for power grid optimization problems: a review of {GPU}-accelerated linear solvers.
\newblock Parallel Computing p. 102870 (2021)

\bibitem{tasseff2019exploring}
Tasseff, B., Coffrin, C., W{\"a}chter, A., Laird, C.: Exploring benefits of linear solver parallelism on modern nonlinear optimization applications.
\newblock arXiv preprint arXiv:1909.08104  (2019)

\bibitem{wachter2006implementation}
W{\"a}chter, A., Biegler, L.T.: On the implementation of an interior-point filter line-search algorithm for large-scale nonlinear programming.
\newblock Mathematical Programming \textbf{106}(1), 25--57 (2006)

\bibitem{waltz2006interior}
Waltz, R.A., Morales, J.L., Nocedal, J., Orban, D.: An interior algorithm for nonlinear optimization that combines line search and trust region steps.
\newblock Mathematical Programming \textbf{107}(3), 391--408 (2006)

\bibitem{wright1998ill}
Wright, M.H.: Ill-conditioning and computational error in interior methods for nonlinear programming.
\newblock SIAM Journal on Optimization \textbf{9}(1), 84--111 (1998)

\bibitem{wright1999modified}
Wright, S.J.: Modified cholesky factorizations in interior-point algorithms for linear programming.
\newblock SIAM Journal on Optimization \textbf{9}(4), 1159--1191 (1999)

\bibitem{wright2001effects}
Wright, S.J.: Effects of finite-precision arithmetic on interior-point methods for nonlinear programming.
\newblock SIAM Journal on Optimization \textbf{12}(1), 36--78 (2001)

\end{thebibliography}
\end{document}